\documentclass{amsart}%
\usepackage{amsmath}
\usepackage{amssymb}
\usepackage{amsfonts}%
\usepackage{graphicx}
\providecommand{\U}[1]{\protect \rule{.1in}{.1in}}
\newtheorem{theorem}{Theorem}
\theoremstyle{plain}

\newtheorem{corollary}{Corollary}

\newtheorem{definition}{Definition}

\newtheorem{lemma}{Lemma}

\newtheorem{remark}{Remark}

\numberwithin{equation}{section}
\begin{document}
\title[generalized vanishing local Morrey estimates]{Sublinear operators with rough kernel generated by fractional integrals and
commutators on generalized vanishing local Morrey spaces}
\author{F.GURBUZ}
\address{ANKARA UNIVERSITY, FACULTY OF SCIENCE, DEPARTMENT OF MATHEMATICS, TANDO\u{G}AN
06100, ANKARA, TURKEY }
\curraddr{}
\email{feritgurbuz84@hotmail.com}
\urladdr{}
\thanks{}
\thanks{}
\thanks{}
\date{}
\dedicatory{ }
\begin{abstract}
In this paper, we consider the norm inequalities for sublinear operators with
rough kernel generated by fractional integrals and commutators on generalized
local Morrey spaces and on generalized vanishing local Morrey spaces including
their weak versions under generic size conditions which are satisfied by most
of the operators in harmonic analysis, respectively. As an example to the
conditions of these theorems are satisfied, we can consider the Marcinkiewicz operator.

\end{abstract}
\maketitle

\section{Introduction}

The classical Morrey spaces $L_{p,\lambda}$ have been introduced by Morrey in
\cite{Morrey} to study the local behavior of solutions of second order
elliptic partial differential equations(PDEs). Later, there are many
applications of Morrey space to the Navier-Stokes equations (see
\cite{Mazzucato}), the Schr\"{o}dinger equations (see \cite{Ruiz}) and the
elliptic problems with discontinuous coefficients (see \cite{Caf, FazPalRag,
Pal}).

Let $B=B(x_{0},r_{B})$ denote the ball with the center $x_{0}$ and radius
$r_{B}$. For a given measurable set $E$, we also denote the Lebesgue measure
of $E$ by $\left \vert E\right \vert $. For any given $\Omega \subseteq
{\mathbb{R}^{n}}$ and $0<p<\infty$, denote by $L_{p}\left(  \Omega \right)  $
the spaces of all functions $f$ satisfying%
\[
\left \Vert f\right \Vert _{L_{p}\left(  \Omega \right)  }=\left(
{\displaystyle \int \limits_{\Omega}}
\left \vert f\left(  x\right)  \right \vert ^{p}dx\right)  ^{\frac{1}{p}}%
<\infty.
\]

We recall the definition of classical Morrey spaces $L_{p,\lambda}$ as%

\[
L_{p,\lambda}\left(  {\mathbb{R}^{n}}\right)  =\left \{  f:\left \Vert
f\right \Vert _{L_{p,\lambda}\left(  {\mathbb{R}^{n}}\right)  }=\sup
\limits_{x\in{\mathbb{R}^{n}},r>0}\,r^{-\frac{\lambda}{p}}\, \Vert
f\Vert_{L_{p}(B(x,r))}<\infty \right \}  ,
\]
where $f\in L_{p}^{loc}({\mathbb{R}^{n}})$, $0\leq \lambda \leq n$ and $1\leq
p<\infty$.

Note that $L_{p,0}=L_{p}({\mathbb{R}^{n}})$ and $L_{p,n}=L_{\infty
}({\mathbb{R}^{n}})$. If $\lambda<0$ or $\lambda>n$, then $L_{p,\lambda
}={\Theta}$, where $\Theta$ is the set of all functions equivalent to $0$ on
${\mathbb{R}^{n}}$.

We also denote by $WL_{p,\lambda}\equiv WL_{p,\lambda}({\mathbb{R}^{n}})$ the
weak Morrey space of all functions $f\in WL_{p}^{loc}({\mathbb{R}^{n}})$ for
which
\[
\left \Vert f\right \Vert _{WL_{p,\lambda}}\equiv \left \Vert f\right \Vert
_{WL_{p,\lambda}({\mathbb{R}^{n}})}=\sup_{x\in{\mathbb{R}^{n}},r>0}%
r^{-\frac{\lambda}{p}}\Vert f\Vert_{WL_{p}(B(x,r))}<\infty,
\]
where $WL_{p}(B(x,r))$ denotes the weak $L_{p}$-space of measurable functions
$f$ for which
\[%
\begin{split}
\Vert f\Vert_{WL_{p}(B(x,r))} &  \equiv \Vert f\chi_{_{B(x,r)}}\Vert
_{WL_{p}({\mathbb{R}^{n}})}\\
&  =\sup_{t>0}t\left \vert \left \{  y\in B(x,r):\,|f(y)|>t\right \}  \right \vert
^{1/{p}}\\
&  =\sup_{0<t\leq|B(x,r)|}t^{1/{p}}\left(  f\chi_{_{B(x,r)}}\right)  ^{\ast
}(t)<\infty,
\end{split}
\]
where $g^{\ast}$ denotes the non-increasing rearrangement of a function $g$.

Throughout the paper we assume that $x\in{\mathbb{R}^{n}}$ and $r>0$ and also
let $B(x,r)$ denotes the open ball centered at $x$ of radius $r$, $B^{C}(x,r)$
denotes its complement and $|B(x,r)|$ is the Lebesgue measure of the ball
$B(x,r)$ and $|B(x,r)|=v_{n}r^{n}$, where $v_{n}=|B(0,1)|$. It is known that
$L_{p,\lambda}({\mathbb{R}^{n}})$ is an extension of $L_{p}({\mathbb{R}^{n}})$
in the sense that $L_{p,0}=L_{p}({\mathbb{R}^{n}})$.

Morrey has stated that many properties of solutions to PDEs can be attributed
to the boundedness of some operators on Morrey spaces. For the boundedness of
the Hardy--Littlewood maximal operator, the fractional integral operator and
the Calder\'{o}n--Zygmund singular integral operator on these spaces, we refer
the readers to \cite{Adams, ChFra, Peetre}. For the properties and
applications of classical Morrey spaces, see \cite{ChFraL1, ChFraL2, FazRag2,
FazPalRag} and references therein.

The study of the operators of harmonic analysis in vanishing Morrey space, in
fact has been almost not touched. A version of the classical Morrey space
$L_{p,\lambda}({\mathbb{R}^{n}})$ where it is possible to approximate by
"nice" functions is the so called vanishing Morrey space $VM_{p,\lambda
}({\mathbb{R}^{n}})$ has been introduced by Vitanza in \cite{Vitanza1} and has
been applied there to obtain a regularity result for elliptic PDEs. This is a
subspace of functions in $L_{p,\lambda}({\mathbb{R}^{n}})$, which satisfies
the condition%
\[
\lim_{r\rightarrow0}\sup_{\underset{0<t<r}{x\in{\mathbb{R}^{n}}}}%
t^{-\frac{\lambda}{p}}\Vert f\Vert_{L_{p}(B(x,t))}=0.
\]
Later in \cite{Vitanza2} Vitanza has proved an existence theorem for a
Dirichlet problem, under weaker assumptions than in \cite{Miranda} and a
$W^{3,2}$ regularity result assuming that the partial derivatives of the
coefficients of the highest and lower order terms belong to vanishing Morrey
spaces depending on the dimension. Also Ragusa has proved a sufficient
condition for commutators of fractional integral operators to belong to
vanishing Morrey spaces $VL_{p,\lambda}({\mathbb{R}^{n}})$
(\cite{PerRagSamWall, RagusaJGlOpt}). For the properties and applications of
vanishing Morrey spaces, see also \cite{Cao-Chen}. It is known that, there is
no research regarding boundedness of the sublinear operators with rough kernel
on vanishing Morrey spaces.

Maximal functions and singular integrals play a key role in harmonic analysis
since maximal functions could control crucial quantitative information
concerning the given functions, despite their larger size, while singular
integrals, Hilbert transform as it's prototype, recently intimately connected
with PDEs, operator theory and other fields.

Let $f\in L^{loc}\left(  {\mathbb{R}^{n}}\right)  $. The
Hardy-Littlewood(H--L) maximal operator $M$ is defined by
\[
Mf(x)=\sup_{t>0}|B(x,t)|^{-1}\int \limits_{B(x,t)}|f(y)|dy.
\]

Let $\overline{T}$ be a standard Calder\'{o}n-Zygmund(C--Z) singular integral
operator, briefly a C--Z operator, i.e., a linear operator bounded from
$L_{2}({\mathbb{R}^{n}})$ to $L_{2}({\mathbb{R}^{n}})$ taking all infinitely
continuously differentiable functions $f$ with compact support to the
functions $f\in L_{1}^{loc}({\mathbb{R}^{n}})$ represented by
\[
\overline{T}f(x)=p.v.\int \limits_{{\mathbb{R}^{n}}}k(x-y)f(y)\,dy\qquad
x\notin suppf.
\]
Such operators have been introduced in \cite{CM}. Here $k$ is a C--Z kernel
\cite{Grafakos}. Chiarenza and Frasca \cite{ChFra} have obtained the
boundedness of H--L maximal operator $M$ and C--Z operator $\overline{T}$ on
$L_{p,\lambda}\left(  {\mathbb{R}^{n}}\right)  $. It is also well known that
H--L maximal operator $M$ and C--Z operator $\overline{T}$ play an important
role in harmonic analysis (see \cite{GarRub, LuDingY, St, Stein93, Torch}).
Also, the theory of the C--Z operator is one of the important achievements of
classical analysis in the last century, which has many important applications
in Fourier analysis, complex analysis, operator theory and so on.

Let $f\in L_{1}^{loc}\left(  {\mathbb{R}^{n}}\right)  $. The fractional
maximal operator $M_{\alpha}$ and the fractional integral operator (also known
as the Riesz potential) $\overline{T}_{\alpha}$ are defined by%

\[
M_{\alpha}f(x)=\sup_{t>0}|B(x,t)|^{-1+\frac{\alpha}{n}}\int \limits_{B(x,t)}%
|f(y)|dy\qquad0\leq \alpha<n
\]

\[
\overline{T}_{\alpha}f\left(  x\right)  =\int \limits_{{\mathbb{R}^{n}}}%
\frac{f\left(  y\right)  }{\left \vert x-y\right \vert ^{n-\alpha}}%
dy\qquad0<\alpha<n.
\]

It is well known that $M_{\alpha}$ and $\overline{T}_{\alpha}$ play an
important role in harmonic analysis (see \cite{Stein93, Torch}).

An early impetus to the study of fractional integrals originated from the
problem of fractional derivation, see e.g. \cite{K-J}. Besides its
contributions to harmonic analysis, fractional integrals also play an
essential role in many other fields. The H-L Sobolev inequality about
fractional integral is still an indispensable tool to establish time-space
estimates for the heat semigroup of nonlinear evolution equations, for some of
this work, see e.g. \cite{KATO}. In recent times, the applications to Chaos
and Fractal have become another motivation to study fractional integrals, see
e.g. \cite{L-SU}. It is well known that $\overline{T}_{\alpha}$ is bounded
from $L_{p}$ to $L_{q}$, where$\frac{1}{p}-$ $\frac{1}{q}=$ $\frac{\alpha}{n}$
and $1<p<\frac{n}{\alpha}$.

Spanne (published by Peetre \cite{Peetre}) and Adams \cite{Adams} have studied
boundedness of the fractional integral operator $\overline{T}_{\alpha}$ on
$L_{p,\lambda}\left(  {\mathbb{R}^{n}}\right)  $. Their results, can be
summarized as follows.

\begin{theorem}
\label{teo1}(Spanne, but published by Peetre \cite{Peetre}) Let $0<\alpha<n$,
$1<p<\frac{n}{\alpha}$, $0<\lambda<n-\alpha p$. Moreover, let $\frac{1}%
{p}-\frac{1}{q}=\frac{\alpha}{n}$ and $\frac{\lambda}{p}=\frac{\mu}{q}$. Then
for $p>1$ the operator $\overline{T}_{\alpha}$ is bounded from $L_{p,\lambda}$
to $L_{q,\lambda}$ and for $p=1$ the operator $\overline{T}_{\alpha}$ is
bounded from $L_{1,\lambda}$ to $WL_{q,\lambda}$.
\end{theorem}

\begin{theorem}
\label{teo2}(Adams \cite{Adams}) Let $0<\alpha<n$, $1<p<\frac{n}{\alpha}$,
$0<\lambda<n-\alpha p$ and $\frac{1}{p}-\frac{1}{q}=\frac{\alpha}{n-\lambda}$.
Then for $p>1$ the operator $\overline{T}_{\alpha}$ is bounded from
$L_{p,\lambda}$ to $L_{q,\lambda}$ and for $p=1$ the operator $\overline
{T}_{\alpha}$ is bounded from $L_{1,\lambda}$ to $WL_{q,\lambda} $.
\end{theorem}

Recall that, for $0<\alpha<n$,%

\[
M_{\alpha}f\left(  x\right)  \leq \nu_{n}^{\frac{\alpha}{n}-1}\overline
{T}_{\alpha}\left(  \left \vert f\right \vert \right)  \left(  x\right)
\]
holds (see \cite{Li-Yang}, Remark 2.1). Hence Theorems \ref{teo1} and
\ref{teo2} also imply boundedness of the fractional maximal operator
$M_{\alpha} $, where $\upsilon_{n}$ is the volume of the unit ball on
${\mathbb{R}^{n}}$.

Suppose that $S^{n-1}$ is the unit sphere on ${\mathbb{R}^{n}}$ $(n\geq2)$
equipped with the normalized Lebesgue measure $d\sigma$. Let $\Omega \in
L_{s}(S^{n-1})$ with $1<s\leq \infty$ be homogeneous of degree zero. We define
$s^{\prime}=\frac{s}{s-1}$ for any $s>1$. Suppose that $T_{\Omega,\alpha}$,
$\alpha \in \left(  0,n\right)  $ represents a linear or a sublinear operator,
which satisfies that for any $f\in L_{1}({\mathbb{R}^{n}})$ with compact
support and $x\notin suppf$
\begin{equation}
|T_{\Omega,\alpha}f(x)|\leq c_{0}\int \limits_{{\mathbb{R}^{n}}}\frac
{|\Omega(x-y)|}{|x-y|^{n-\alpha}}\,|f(y)|\,dy,\label{e1}%
\end{equation}
where $c_{0}$ is independent of $f$ and $x$.

We point out that the condition (\ref{e1}) in the case of $\Omega \equiv1$,
$\alpha=0$ has been introduced by Soria and Weiss in \cite{SW}. The condition
(\ref{e1}) is satisfied by many interesting operators in harmonic analysis,
such as fractional maximal operator, fractional integral operator(Riesz
potential), fractional Marcinkiewicz operator and so on (see \cite{LLY},
\cite{SW} for details).

In 1971, Muckenhoupt and Wheeden \cite{Muckenhoupt and Wheeden} defined the
fractional integral operator with rough kernel $\overline{T}_{\Omega,\alpha}$ by%

\[
\overline{T}_{\Omega,\alpha}f(x)=\int \limits_{{\mathbb{R}^{n}}}\frac
{\Omega(x-y)}{|x-y|^{n-\alpha}}f(y)dy\qquad0<\alpha<n
\]
and a related fractional maximal operator with rough kernel $M_{\Omega,\alpha
}$ is given by%

\[
M_{\Omega,\alpha}f(x)=\sup_{t>0}|B(x,t)|^{-1+\frac{\alpha}{n}}\int
\limits_{B(x,t)}\left \vert \Omega \left(  x-y\right)  \right \vert
|f(y)|dy\qquad0\leq \alpha<n,
\]
where $\Omega \in L_{s}(S^{n-1})$ with $1<s\leq \infty$ is homogeneous of degree
zero on ${\mathbb{R}^{n}}$ and $\overline{T}_{\Omega,\alpha}$ satisfies the
condition (\ref{e1}).

If $\alpha=0$, then $M_{\Omega,0}\equiv M_{\Omega}$ H-L maximal operator with
rough kernel. It is obvious that when $\Omega \equiv1$, $M_{1,\alpha}\equiv
M_{\alpha}$ and $\overline{T}_{1,\alpha}\equiv \overline{T}_{\alpha}$ are the
fractional maximal operator and the fractional integral operator, respectively.

In recent years, the mapping properties of $\overline{T}_{\Omega,\alpha}$ on
some kinds of function spaces have been studied in many papers (see
\cite{ChanilloWW}, \cite{DingLu2}, \cite{DingLu3}, \cite{Muckenhoupt and
Wheeden} for details). In particular, the boundedness of $\overline{T}%
_{\Omega,\alpha}$ in Lebesgue spaces has been obtained.

\begin{lemma}
\label{Lemma1}$\left(  \text{\cite{ChanilloWW, DingLu2, Muckenhoupt}}\right)
$ Let $0<\alpha<n$, $1<p<\frac{n}{\alpha}$ and $\frac{1}{q}=\frac{1}{p}%
-\frac{\alpha}{n}$. If $\Omega \in L_{s}(S^{n-1})$, $s>\frac{n}{n-\alpha}$,
then we have
\end{lemma}%

\[
\left \Vert \overline{T}_{\Omega,\alpha}f\right \Vert _{L_{q}}\leq C\left \Vert
f\right \Vert _{L_{p}}.
\]

\begin{corollary}
\label{Corollary0*}Under the assumptions of Lemma \ref{Lemma1}, the operator
$M_{\Omega,\alpha}$ is bounded from $L_{p}({\mathbb{R}^{n}})$ to
$L_{q}({\mathbb{R}^{n}})$. Moreover, we have%

\[
\left \Vert M_{\Omega,\alpha}f\right \Vert _{L_{q}}\leq C\left \Vert f\right \Vert
_{L_{p}}.
\]

\end{corollary}

\begin{proof}
Set%
\[
\widetilde{T}_{\left \vert \Omega \right \vert ,\alpha}\left(  \left \vert
f\right \vert \right)  (x)=\int \limits_{{\mathbb{R}^{n}}}\frac{\left \vert
\Omega(x-y)\right \vert }{|x-y|^{n-\alpha}}\left \vert f(y)\right \vert
dy\qquad0<\alpha<n,
\]
where $\Omega \in L_{s}(S^{n-1})\left(  s>1\right)  $ is homogeneous of degree
zero on ${\mathbb{R}^{n}}$. It is easy to see that, for $\widetilde
{T}_{\left \vert \Omega \right \vert ,\alpha}$, Lemma \ref{Lemma1} is also hold.
On the other hand, for any $t>0$, we have
\begin{align*}
\widetilde{T}_{\left \vert \Omega \right \vert ,\alpha}\left(  \left \vert
f\right \vert \right)  (x)  & \geq%
{\displaystyle \int \limits_{B\left(  x,t\right)  }}
\frac{\left \vert \Omega(x-y)\right \vert }{|x-y|^{n-\alpha}}\left \vert
f(y)\right \vert dy\\
& \geq \frac{1}{t^{n-\alpha}}%
{\displaystyle \int \limits_{B\left(  x,t\right)  }}
\left \vert \Omega(x-y)\right \vert \left \vert f(y)\right \vert dy.
\end{align*}
Taking the supremum for $t>0$ on the inequality above, we get%
\[
M_{\Omega,\alpha}f\left(  x\right)  \leq C_{n,\alpha}^{-1}\widetilde
{T}_{\left \vert \Omega \right \vert ,\alpha}\left(  \left \vert f\right \vert
\right)  (x)\qquad C_{n,\alpha}=\left \vert B\left(  0,1\right)  \right \vert
^{\frac{n-\alpha}{n}}.
\]

\end{proof}

In 1976, Coifman, Rocherberg and Weiss \cite{CRW} introduced the commutator
ge-nerated by $\overline{T}_{\Omega}$ and a local integrable function $b$:
\begin{equation}
\lbrack b,\overline{T}_{\Omega}]f(x)\equiv b(x)\overline{T}_{\Omega
}f(x)-\overline{T}_{\Omega}(bf)(x)=p.v.\int \limits_{{\mathbb{R}^{n}}%
}[b(x)-b(y)]\frac{\Omega(x-y)}{|x-y|^{n}}f(y)dy.\label{e5}%
\end{equation}
Sometimes, the commutator defined by (\ref{e5}) is also called the commutator
in Coifman-Rocherberg-Weiss's sense, which has its root in the complex
analysis and harmonic analysis (see \cite{CRW}).

Let $b$ be a locally integrable function on ${\mathbb{R}^{n}}$, then we shall
define the commutators generated by fractional integral operators with rough
kernel and $b$ as follows.
\[
\lbrack b,\overline{T}_{\Omega,\alpha}]f(x)\equiv b(x)\overline{T}%
_{\Omega,\alpha}f(x)-\overline{T}_{\Omega,\alpha}(bf)(x)=p.v.\int
\limits_{{\mathbb{R}^{n}}}[b(x)-b(y)]\frac{\Omega(x-y)}{|x-y|^{n-\alpha}%
}f(y)dy,
\]
where $0<\alpha<n$, and $f$ is a suitable function.

\begin{remark}
\cite{Shi} When $\Omega$ satisfies the specified size conditions, the kernel
of the operator $\overline{T}_{\Omega,\alpha}$ has no regularity, so the
operator $\overline{T}_{\Omega,\alpha}$ is called a rough fractional integral
operator. In recent years, a variety of operators related to the fractional
integrals, but lacking the smoothness required in the classical theory, have
been studied. These include the operator $[b,\overline{T}_{\Omega,\alpha}]$.
For more results, we refer the reader to \cite{Chanillo, DingLu2, DingLu3,
FuLinLu, GulJMS2013, Gurbuz, Gurbuz2, Gurbuz3, Yu}.
\end{remark}

In this paper, extending the definition of vanishing Morrey spaces
\cite{Vitanza1} and vanishing generalized Morrey spaces \cite{N. Samko}, the
author introduces the generalized vanishing local Morrey spaces
$VLM_{p,\varphi}^{\left \{  x_{0}\right \}  }$, including their weak versions
and studies the boundedness of the sublinear operators with rough kernel
generated by fractional integrals and commutators in these spaces. These
conditions are satisfied by most of the operators in harmonic analysis, such
as fractional maximal operator, fractional integral operator(Riesz potential),
fractional Marcinkiewicz operator and so on. In all the cases the conditions
for the boundedness of $T_{\Omega,\alpha}$ and $T_{\Omega,b,\alpha}$ are given
in terms of Zygmund-type integral inequalities on $\left(  \varphi_{1}%
,\varphi_{2}\right)  $, where there is no assumption on monotonicity of
$\varphi_{1},\varphi_{2}$ in $r$. As an example to the conditions of these
theorems are satisfied, we can consider the Marcinkiewicz operator.

By $A\lesssim B$ we mean that $A\leq CB$ with some positive constant $C$
independent of appropriate quantities. If $A\lesssim B$ and $B\lesssim A$, we
write $A\approx B$ and say that $A$ and $B$ are equivalent.

\section{generalized vanishing local Morrey spaces}

After studying Morrey spaces in detail, researchers have passed to generalized
Morrey spaces. Mizuhara \cite{Miz} has given generalized Morrey spaces
$M_{p,\varphi}$ considering $\varphi=\varphi \left(  r\right)  $ instead of
$r^{\lambda}$ in the above definition of the Morrey space. Later, Guliyev
\cite{GulJIA} has defined the generalized Morrey spaces $M_{p,\varphi}$ with
normalized norm as follows:

\begin{definition}
\cite{GulJIA} \textbf{(generalized Morrey space) }Let $\varphi(x,r)$ be a
positive measurable function on ${\mathbb{R}^{n}}\times(0,\infty)$ and $1\leq
p<\infty$. We denote by $M_{p,\varphi}\equiv M_{p,\varphi}({\mathbb{R}^{n}})$
the generalized Morrey space, the space of all functions $f\in L_{p}%
^{loc}({\mathbb{R}^{n}})$ with finite quasinorm
\[
\Vert f\Vert_{M_{p,\varphi}}=\sup \limits_{x\in{\mathbb{R}^{n}},r>0}%
\varphi(x,r)^{-1}\,|B(x,r)|^{-\frac{1}{p}}\, \Vert f\Vert_{L_{p}(B(x,r))}.
\]
Also by $WM_{p,\varphi}\equiv WM_{p,\varphi}({\mathbb{R}^{n}})$ we denote the
weak generalized Morrey space of all functions $f\in WL_{p}^{loc}%
({\mathbb{R}^{n}})$ for which
\[
\Vert f\Vert_{WM_{p,\varphi}}=\sup \limits_{x\in{\mathbb{R}^{n}},r>0}%
\varphi(x,r)^{-1}\,|B(x,r)|^{-\frac{1}{p}}\, \Vert f\Vert_{WL_{p}%
(B(x,r))}<\infty.
\]

\end{definition}

According to this definition, we recover the Morrey space $L_{p,\lambda}$ and
weak Morrey space $WL_{p,\lambda}$ under the choice $\varphi(x,r)=r^{\frac
{\lambda-n}{p}}$:
\[
L_{p,\lambda}=M_{p,\varphi}\mid_{\varphi(x,r)=r^{\frac{\lambda-n}{p}}%
},~~~~~~~~WL_{p,\lambda}=WM_{p,\varphi}\mid_{\varphi(x,r)=r^{\frac{\lambda
-n}{p}}}.
\]

During the last decades various classical operators, such as maximal, singular
and potential operators have been widely investigated in generalized Morrey
spaces (see \cite{Ding, GulJIA, Gurbuz, Karaman, Softova} for details).

Recall that in 2015 the work \cite{BGGS} and the Ph.D. thesis \cite{Gurbuz} by
Gurbuz et al. have been introduced the generalized local Morrey space
$LM_{p,\varphi}^{\{x_{0}\}}$ given by

\begin{definition}
\label{Definition2}\textbf{(generalized local Morrey space) }Let
$\varphi(x,r)$ be a positive measurable function on ${\mathbb{R}^{n}}%
\times(0,\infty)$ and $1\leq p<\infty$. For any fixed $x_{0}\in{\mathbb{R}%
^{n}}$ we denote by $LM_{p,\varphi}^{\{x_{0}\}}\equiv LM_{p,\varphi}%
^{\{x_{0}\}}({\mathbb{R}^{n}})$ the generalized local Morrey space, the space
of all functions $f\in L_{p}^{loc}({\mathbb{R}^{n}})$ with finite quasinorm
\[
\Vert f\Vert_{LM_{p,\varphi}^{\{x_{0}\}}}=\sup \limits_{r>0}\varphi
(x_{0},r)^{-1}\,|B(x_{0},r)|^{-\frac{1}{p}}\, \Vert f\Vert_{L_{p}(B(x_{0}%
,r))}<\infty.
\]
Also by $WLM_{p,\varphi}^{\{x_{0}\}}\equiv WLM_{p,\varphi}^{\{x_{0}%
\}}({\mathbb{R}^{n}})$ we denote the weak generalized local Morrey space of
all functions $f\in WL_{p}^{loc}({\mathbb{R}^{n}})$ for which
\[
\Vert f\Vert_{WLM_{p,\varphi}^{\{x_{0}\}}}=\sup \limits_{r>0}\varphi
(x_{0},r)^{-1}\,|B(x_{0},r)|^{-\frac{1}{p}}\, \Vert f\Vert_{WL_{p}%
(B(x_{0},r))}<\infty.
\]

\end{definition}

According to this definition, we recover the local Morrey space $LM_{p,\lambda
}^{\{x_{0}\}}$ and the weak local Morrey space $WLM_{p,\lambda}^{\{x_{0}\}}$
under the choice $\varphi(x_{0},r)=r^{\frac{\lambda-n}{p}}$:
\[
LL_{p,\lambda}^{\{x_{0}\}}=LM_{p,\varphi}^{\{x_{0}\}}\mid_{\varphi
(x_{0},r)=r^{\frac{\lambda-n}{p}}},~~~~~~WLL_{p,\lambda}^{\{x_{0}%
\}}=WLM_{p,\varphi}^{\{x_{0}\}}\mid_{\varphi(x_{0},r)=r^{\frac{\lambda-n}{p}}%
}.
\]

The main goal of \cite{BGGS} and \cite{Gurbuz} is to give some sufficient
conditions for the boundedness of a large class of rough sublinear operators
and their commutators on the generalized local Morrey space $LM_{p,\varphi
}^{\{x_{0}\}}$. For the properties and applications of generalized local
Morrey spaces $LM_{p,\varphi}^{\{x_{0}\}}$, see also \cite{Gurbuz}.

Furthermore, we have the following embeddings:%
\[
M_{p,\varphi}\subset LM_{p,\varphi}^{\{x_{0}\}},\qquad \Vert f\Vert
_{LM_{p,\varphi}^{\{x_{0}\}}}\leq \Vert f\Vert_{M_{p,\varphi}},
\]%
\[
WM_{p,\varphi}\subset WLM_{p,\varphi}^{\{x_{0}\}},\qquad \Vert f\Vert
_{WLM_{p,\varphi}^{\{x_{0}\}}}\leq \Vert f\Vert_{WM_{p,\varphi}}.
\]
Wiener \cite{Wiener1, Wiener2} has looked for a way to describe the behavior
of a function at the infinity. The conditions he considered are related to
appropriate weighted $L_{q}$ spaces. Beurling \cite{Beurl} has extended this
idea and has defined a pair of dual Banach spaces $A_{q}$ and $B_{q^{\prime}}%
$, where $1/q+1/q^{\prime}=1$. To be precise, $A_{q}$ is a Banach algebra with
respect to the convolution, expressed as a union of certain weighted $L_{q}$
spaces; the space $B_{q^{\prime}}$ is expressed as the intersection of the
corresponding weighted $L_{q^{\prime}}$ spaces. Feichtinger \cite{Feicht} has
observed that the space $B_{q}$ can be described by
\begin{equation}
\left \Vert f\right \Vert _{B_{q}}=\sup_{k\geq0}2^{-\frac{kn}{q}}\Vert f\chi
_{k}\Vert_{L_{q}({\mathbb{R}^{n}})}<\infty,\label{e21}%
\end{equation}
where $\chi_{0}$ is the characteristic function of the unit ball
$\{x\in{\mathbb{R}^{n}}:|x|\leq1\}$, $\chi_{k}$ is the characteristic function
of the annulus $\{x\in{\mathbb{R}^{n}}:2^{k-1}<|x|\leq2^{k}\}$, $k=1,2,\ldots
$. By duality, the space $A_{q}({\mathbb{R}^{n}})$, appropriately called now
the Beurling algebra , can be described by the condition%

\begin{equation}
\left \Vert f\right \Vert _{A_{q}}=\sum \limits_{k=0}^{\infty}2^{\frac
{kn}{q^{\prime}}}\Vert f\chi_{k}\Vert_{L_{q}({\mathbb{R}^{n}})}<\infty
.\label{e22}%
\end{equation}

Let $\dot{B}_{q}({\mathbb{R}^{n}})$ and $\dot{A}_{q}({\mathbb{R}^{n}})$ be the
homogeneous versions of $B_{q}({\mathbb{R}^{n}})$ and $A_{q}({\mathbb{R}^{n}%
})$ by taking $k\in \mathbb{Z}$ in (\ref{e21}) and (\ref{e22}) instead of
$k\geq0$ there.

If $\lambda<0$ or $\lambda>n$, then $LM_{p,\lambda}^{\{x_{0}\}}({\mathbb{R}%
^{n}})={\Theta}$, where $\Theta$ is the set of all functions equivalent to $0$
on ${\mathbb{R}^{n}}$. Note that $LM_{p,0}({\mathbb{R}^{n}})=L_{p}%
({\mathbb{R}^{n}})$ and $LM_{p,n}({\mathbb{R}^{n}})=\dot{B}_{p}({\mathbb{R}%
^{n}})$.
\[
\dot{B}_{p,\mu}=LM_{p,\varphi}\mid_{\varphi(0,r)=r^{\mu n}},~~~~~~W\dot
{B}_{p,\mu}=WLM_{p,\varphi}\mid_{\varphi(0,r)=r^{\mu n}}.
\]

Alvarez et al. \cite{AlvLanLakey}, in order to study the relationship between
central $BMO$ spaces and Morrey spaces, they introduced $\lambda$-central
bounded mean oscillation spaces and central Morrey spaces $\dot{B}_{p,\mu
}({\mathbb{R}^{n}})\equiv LM_{p,n+np\mu}({\mathbb{R}^{n}})$, $\mu \in
\lbrack-\frac{1}{p},0]$. If $\mu<-\frac{1}{p}$ or $\mu>0$, then $\dot
{B}_{p,\mu}({\mathbb{R}^{n}})={\Theta}$. Note that $\dot{B}_{p,-\frac{1}{p}%
}({\mathbb{R}^{n}})=L_{p}({\mathbb{R}^{n}})$ and $\dot{B}_{p,0}({\mathbb{R}%
^{n}})=\dot{B}_{p}({\mathbb{R}^{n}})$. Also define the weak central Morrey
spaces $W\dot{B}_{p,\mu}({\mathbb{R}^{n}})\equiv WLM_{p,n+np\mu}%
({\mathbb{R}^{n}})$.

The vanishing generalized Morrey spaces $VM_{p,\varphi}({\mathbb{R}^{n}})$
which has been introduced and studied by Samko \cite{N. Samko} is defined as follows.

\begin{definition}
\textbf{(vanishing generalized Morrey space) }Let $\varphi(x,r)$ be a positive
measurable function on ${\mathbb{R}^{n}}\times(0,\infty)$ and $1\leq p<\infty
$. The vanishing generalized Morrey space $VM_{p,\varphi}({\mathbb{R}^{n}})$
is defined as the spaces of functions $f\in L_{p}^{loc}({\mathbb{R}^{n}})$
such that%
\[
\lim \limits_{r\rightarrow0}\sup \limits_{x\in{\mathbb{R}^{n}}}\varphi
(x,r)^{-1}\, \int_{B(x,r)}|f(y)|^{p}dy=0.
\]

\end{definition}

Everywhere in the sequel we assume that%
\[
\lim_{t\rightarrow0}\frac{t^{\frac{n}{p}}}{\varphi(x,t)}=0,
\]
and%
\[
\sup_{0<t<\infty}\frac{t^{\frac{n}{p}}}{\varphi(x,t)}<\infty,
\]
which make the spaces $VM_{p,\varphi}({\mathbb{R}^{n}})$ non-trivial, because
bounded functions with compact support belong to this space. The spaces
$VM_{p,\varphi}({\mathbb{R}^{n}})$ and $WVM_{p,\varphi}({\mathbb{R}^{n}})$ are
Banach spaces with respect to the norm (see, for example \cite{N. Samko})%
\[
\Vert f\Vert_{VM_{p,\varphi}}=\sup \limits_{x\in{\mathbb{R}^{n}},r>0}%
\varphi(x,r)^{-1}\Vert f\Vert_{L_{p}(B(x,r))},
\]%
\[
\Vert f\Vert_{WVM_{p,\varphi}}=\sup \limits_{x\in{\mathbb{R}^{n}},r>0}%
\varphi(x,r)^{-1}\Vert f\Vert_{WL_{p}(B(x,r)),}%
\]
respectively. The spaces $VM_{p,\varphi}({\mathbb{R}^{n}})$ and
$WVM_{p,\varphi}({\mathbb{R}^{n}})$ are closed subspaces of the Banach spaces
$M_{p,\varphi}({\mathbb{R}^{n}})$ and $WM_{p,\varphi}({\mathbb{R}^{n}})$,
respectively, which may be shown by standard means.

Furthermore, we have the following embeddings:%
\[
VM_{p,\varphi}\subset M_{p,\varphi},\qquad \Vert f\Vert_{M_{p,\varphi}}%
\leq \Vert f\Vert_{VM_{p,\varphi}},
\]%
\[
WVM_{p,\varphi}\subset WM_{p,\varphi},\qquad \Vert f\Vert_{WM_{p,\varphi}}%
\leq \Vert f\Vert_{WVM_{p,\varphi}}.
\]

For the properties and applications of vanishing generalized Morrey spaces,
see also \cite{Akbulut-Kuzu, N. Samko}.

For brevity, in the sequel we use the notations%
\[
\mathfrak{M}_{p,\varphi}\left(  f;x_{0},r\right)  :=\frac{|B(x_{0}%
,r)|^{-\frac{1}{p}}\, \Vert f\Vert_{L_{p}(B(x_{0},r))}}{\varphi(x_{0},r)}%
\]
and%
\[
\mathfrak{M}_{p,\varphi}^{W}\left(  f;x_{0},r\right)  :=\frac{|B(x_{0}%
,r)|^{-\frac{1}{p}}\, \Vert f\Vert_{WL_{p}(B(x_{0},r))}}{\varphi(x_{0},r)}.
\]

Extending the definition of vanishing generalized Morrey spaces to the case of
generalized local Morrey spaces, we introduce the following definitions.

\begin{definition}
\textbf{(generalized vanishing local Morrey space) }The generalized vanishing
local Morrey space $VLM_{p,\varphi}^{\left \{  x_{0}\right \}  }({\mathbb{R}%
^{n}})$ is defined as the spaces of functions $f\in LM_{p,\varphi}^{\{x_{0}%
\}}({\mathbb{R}^{n}})$ such that%
\begin{equation}
\lim \limits_{r\rightarrow0}\mathfrak{M}_{p,\varphi}\left(  f;x_{0},r\right)
=0.\label{1*}%
\end{equation}

\end{definition}

\begin{definition}
\textbf{(weak generalized vanishing local Morrey space) }The weak generalized
vanishing local Morrey space $WVLM_{p,\varphi}^{\left \{  x_{0}\right \}
}({\mathbb{R}^{n}})$ is defined as the spaces of functions $f\in
WLM_{p,\varphi}^{\{x_{0}\}}({\mathbb{R}^{n}})$ such that%
\begin{equation}
\lim \limits_{r\rightarrow0}\mathfrak{M}_{p,\varphi}^{W}\left(  f;x_{0}%
,r\right)  =0.\label{2*}%
\end{equation}

\end{definition}

Everywhere in the sequel we assume that%
\begin{equation}
\lim_{r\rightarrow0}\frac{1}{\varphi(x_{0},r)}=0,\label{2}%
\end{equation}
and%
\begin{equation}
\sup_{0<r<\infty}\frac{1}{\varphi(x_{0},r)}<\infty,\label{3}%
\end{equation}
which make the spaces $VLM_{p,\varphi}^{\left \{  x_{0}\right \}  }%
({\mathbb{R}^{n}})$ non-trivial, because bounded functions with compact
support belong to this space. The spaces $VLM_{p,\varphi}^{\left \{
x_{0}\right \}  }({\mathbb{R}^{n}})$ and $WVLM_{p,\varphi}^{\{x_{0}%
\}}({\mathbb{R}^{n}})$ are Banach spaces with respect to the norm
\begin{equation}
\Vert f\Vert_{VLM_{p,\varphi}^{\left \{  x_{0}\right \}  }}\equiv \Vert
f\Vert_{LM_{p,\varphi}^{\{x_{0}\}}}=\sup \limits_{r>0}\mathfrak{M}_{p,\varphi
}\left(  f;x_{0},r\right)  ,\label{4}%
\end{equation}%
\begin{equation}
\Vert f\Vert_{WVLM_{p,\varphi}^{\left \{  x_{0}\right \}  }}=\Vert
f\Vert_{WLM_{p,\varphi}^{\{x_{0}\}}}=\sup \limits_{r>0}\mathfrak{M}_{p,\varphi
}^{W}\left(  f;x_{0},r\right)  ,\label{5}%
\end{equation}
respectively. The spaces $VLM_{p,\varphi}^{\left \{  x_{0}\right \}
}({\mathbb{R}^{n}})$ and $WVLM_{p,\varphi}^{\{x_{0}\}}({\mathbb{R}^{n}})$ are
closed subspaces of the Banach spaces $LM_{p,\varphi}^{\left \{  x_{0}\right \}
}({\mathbb{R}^{n}})$ and $WLM_{p,\varphi}^{\{x_{0}\}}({\mathbb{R}^{n}})$,
respectively, which may be shown by standard means.

\section{Sublinear operators with rough kernel $T_{\Omega,\alpha}$ on the
spaces $LM_{p,\varphi}^{\left \{  x_{0}\right \}  }$ and $VLM_{p,\varphi
}^{\left \{  x_{0}\right \}  }$}

In this section, we will first prove the boundedness of the operator
$T_{\Omega,\alpha}$ satisfying (\ref{e1}) on the generalized local Morrey
spaces $LM_{p,\varphi}^{\left \{  x_{0}\right \}  }({\mathbb{R}^{n}})$,
including their weak versions by using the following statement on the
boundedness of the weighted Hardy operator
\[
H_{\omega}g(t):=\int \limits_{t}^{\infty}g(s)\omega(s)ds,\qquad0<t<\infty,
\]
where $\omega$ is a fixed non-negative function and measurable on $(0,\infty
)$. Then, We will also give the boundedness of $T_{\Omega,\alpha} $ satisfying
(\ref{e1}) on generalized vanishing local Morrey spaces $VLM_{p,\varphi
}^{\left \{  x_{0}\right \}  }({\mathbb{R}^{n}})$, including their weak versions.

\begin{theorem}
\label{teo5}\cite{BGGS, GulJMS2013, Gurbuz, Gurbuz2, Gurbuz3} Let $v_{1}$,
$v_{2}$ and $\omega$ be positive almost everywhere and measurable functions on
$(0,\infty)$. The inequality
\begin{equation}
\operatorname*{esssup}\limits_{t>0}v_{2}(t)H_{\omega}g(t)\leq
C\operatorname*{esssup}\limits_{t>0}v_{1}(t)g(t)\label{e31}%
\end{equation}
holds for some $C>0$ for all non-negative and non-decreasing functions $g$ on
$(0,\infty)$ if and only if
\begin{equation}
B:=\sup \limits_{t>0}v_{2}(t)\int \limits_{t}^{\infty}\frac{\omega
(s)ds}{\operatorname*{esssup}\limits_{s<\tau<\infty}v_{1}(\tau)}%
<\infty.\label{3*}%
\end{equation}

Moreover, the value $C=B$ is the best constant for (\ref{e31}).
\end{theorem}

We first prove the following Lemma \ref{lemma2}.

\begin{lemma}
\label{lemma2} Suppose that $x_{0}\in{\mathbb{R}^{n}}$, $\Omega \in
L_{s}(S^{n-1})$, $1<s\leq \infty$, is homogeneous of degree zero. Let
$0<\alpha<n$, $1\leq p<\frac{n}{\alpha}$, $\frac{1}{q}=\frac{1}{p}%
-\frac{\alpha}{n}$. Let $T_{\Omega,\alpha}$ be a sublinear operator satisfying
condition (\ref{e1}), bounded from $L_{p}({\mathbb{R}^{n}})$ to $L_{q}%
({\mathbb{R}^{n}})$ for $p>1$, and bounded from $L_{1}({\mathbb{R}^{n}})$ to
$WL_{q}({\mathbb{R}^{n}})$ for $p=1$.

\textit{If }$p>1$\textit{\ and }$s^{\prime}\leq p$\textit{, then the
inequality}%
\begin{equation}
\left \Vert T_{\Omega,\alpha}f\right \Vert _{L_{q}\left(  B\left(
x_{0},r\right)  \right)  }\lesssim r^{\frac{n}{q}}\int \limits_{2r}^{\infty
}t^{-\frac{n}{q}-1}\left \Vert f\right \Vert _{L_{p}\left(  B\left(
x_{0},t\right)  \right)  }dt\label{100}%
\end{equation}
holds for any ball $B\left(  x_{0},r\right)  $ and for all $f\in L_{p}%
^{loc}\left(  {\mathbb{R}^{n}}\right)  $.

If $p>1$ and $q<s$, then the inequality%
\[
\left \Vert T_{\Omega,\alpha}f\right \Vert _{L_{q}\left(  B\left(
x_{0},r\right)  \right)  }\lesssim r^{\frac{n}{q}-\frac{n}{s}}\int
\limits_{2r}^{\infty}t^{\frac{n}{s}-\frac{n}{q}-1}\left \Vert f\right \Vert
_{L_{p}\left(  B\left(  x_{0},t\right)  \right)  }dt
\]
holds for any ball $B\left(  x_{0},r\right)  $ and for all $f\in L_{p}%
^{loc}\left(  {\mathbb{R}^{n}}\right)  $.

Moreover, for $p=1<q<s$ the inequality%
\begin{equation}
\left \Vert T_{\Omega,\alpha}f\right \Vert _{WL_{q}\left(  B\left(
x_{0},r\right)  \right)  }\lesssim r^{\frac{n}{q}}\int \limits_{2r}^{\infty
}t^{-\frac{n}{q}-1}\left \Vert f\right \Vert _{L_{1}\left(  B\left(
x_{0},t\right)  \right)  }dt\label{e38}%
\end{equation}
holds for any ball $B\left(  x_{0},r\right)  $ and for all $f\in L_{1}%
^{loc}\left(  {\mathbb{R}^{n}}\right)  $.
\end{lemma}

\begin{proof}
Let $0<\alpha<n$, $1\leq s^{\prime}<p<\frac{n}{\alpha}$ and $\frac{1}{q}%
=\frac{1}{p}-\frac{\alpha}{n}$. Set $B=B\left(  x_{0},r\right)  $ for the ball
centered at $x_{0}$ and of radius $r$ and $2B=B\left(  x_{0},2r\right)  $. We
represent $f$ as%
\begin{equation}
f=f_{1}+f_{2},\qquad \text{\ }f_{1}\left(  y\right)  =f\left(  y\right)
\chi_{2B}\left(  y\right)  ,\qquad \text{\ }f_{2}\left(  y\right)  =f\left(
y\right)  \chi_{\left(  2B\right)  ^{C}}\left(  y\right)  ,\qquad
r>0\label{e39}%
\end{equation}
and have%
\[
\left \Vert T_{\Omega,\alpha}f\right \Vert _{L_{q}\left(  B\right)  }%
\leq \left \Vert T_{\Omega,\alpha}f_{1}\right \Vert _{L_{q}\left(  B\right)
}+\left \Vert T_{\Omega,\alpha}f_{2}\right \Vert _{L_{q}\left(  B\right)  }.
\]

Since $f_{1}\in L_{p}\left(  \mathbb{R}^{n}\right)  $, $T_{\Omega,\alpha}%
f_{1}\in L_{q}\left(  \mathbb{R}^{n}\right)  $ and from the boundedness of
$T_{\Omega,\alpha}$ from $L_{p}({\mathbb{R}^{n}})$ to $L_{q}({\mathbb{R}^{n}%
})$ (see Lemma \ref{Lemma1}) it follows that:%
\[
\left \Vert T_{\Omega,\alpha}f_{1}\right \Vert _{L_{q}\left(  B\right)  }%
\leq \left \Vert T_{\Omega,\alpha}f_{1}\right \Vert _{L_{q}\left(
\mathbb{R}
^{n}\right)  }\leq C\left \Vert f_{1}\right \Vert _{L_{p}\left(
\mathbb{R}
^{n}\right)  }=C\left \Vert f\right \Vert _{L_{p}\left(  2B\right)  },
\]
where constant $C>0$ is independent of $f$.

It is clear that $x\in B$, $y\in \left(  2B\right)  ^{C}$ implies \ $\frac
{1}{2}\left \vert x_{0}-y\right \vert \leq \left \vert x-y\right \vert \leq \frac
{3}{2}\left \vert x_{0}-y\right \vert $. We get%
\[
\left \vert T_{\Omega,\alpha}f_{2}\left(  x\right)  \right \vert \leq
2^{n-\alpha}c_{1}\int \limits_{\left(  2B\right)  ^{C}}\frac{\left \vert
f\left(  y\right)  \right \vert \left \vert \Omega \left(  x-y\right)
\right \vert }{\left \vert x_{0}-y\right \vert ^{n-\alpha}}dy.
\]

By the Fubini's theorem, we have%
\begin{align*}
\int \limits_{\left(  2B\right)  ^{C}}\frac{\left \vert f\left(  y\right)
\right \vert \left \vert \Omega \left(  x-y\right)  \right \vert }{\left \vert
x_{0}-y\right \vert ^{n-\alpha}}dy  & \approx \int \limits_{\left(  2B\right)
^{C}}\left \vert f\left(  y\right)  \right \vert \left \vert \Omega \left(
x-y\right)  \right \vert \int \limits_{\left \vert x_{0}-y\right \vert }^{\infty
}\frac{dt}{t^{n+1-\alpha}}dy\\
& \approx \int \limits_{2r}^{\infty}\int \limits_{2r\leq \left \vert x_{0}%
-y\right \vert \leq t}\left \vert f\left(  y\right)  \right \vert \left \vert
\Omega \left(  x-y\right)  \right \vert dy\frac{dt}{t^{n+1-\alpha}}\\
& \lesssim \int \limits_{2r}^{\infty}\int \limits_{B\left(  x_{0},t\right)
}\left \vert f\left(  y\right)  \right \vert \left \vert \Omega \left(
x-y\right)  \right \vert dy\frac{dt}{t^{n+1-\alpha}}.
\end{align*}

Applying the H\"{o}lder's inequality, we get%
\begin{align}
& \int \limits_{\left(  2B\right)  ^{C}}\frac{\left \vert f\left(  y\right)
\right \vert \left \vert \Omega \left(  x-y\right)  \right \vert }{\left \vert
x_{0}-y\right \vert ^{n-\alpha}}dy\nonumber \\
& \lesssim \int \limits_{2r}^{\infty}\left \Vert f\right \Vert _{L_{p}\left(
B\left(  x_{0},t\right)  \right)  }\left \Vert \Omega \left(  x-\cdot \right)
\right \Vert _{L_{s}\left(  B\left(  x_{0},t\right)  \right)  }\left \vert
B\left(  x_{0},t\right)  \right \vert ^{1-\frac{1}{p}-\frac{1}{s}}\frac
{dt}{t^{n+1-\alpha}}.\label{e310}%
\end{align}

For $x\in B\left(  x_{0},t\right)  $, notice that $\Omega$ is homogenous of
degree zero and $\Omega \in L_{s}(S^{n-1})$, $s>1$. Then, we obtain%
\begin{align}
\left(  \int \limits_{B\left(  x_{0},t\right)  }\left \vert \Omega \left(
x-y\right)  \right \vert ^{s}dy\right)  ^{\frac{1}{s}}  & =\left(
\int \limits_{B\left(  x-x_{0},t\right)  }\left \vert \Omega \left(  z\right)
\right \vert ^{s}dz\right)  ^{\frac{1}{s}}\nonumber \\
& \leq \left(  \int \limits_{B\left(  0,t+\left \vert x-x_{0}\right \vert \right)
}\left \vert \Omega \left(  z\right)  \right \vert ^{s}dz\right)  ^{\frac{1}{s}%
}\nonumber \\
& \leq \left(  \int \limits_{B\left(  0,2t\right)  }\left \vert \Omega \left(
z\right)  \right \vert ^{s}dz\right)  ^{\frac{1}{s}}\nonumber \\
& =\left(  \int \limits_{S^{n-1}}\int \limits_{0}^{2t}\left \vert \Omega \left(
z^{\prime}\right)  \right \vert ^{s}d\sigma \left(  z^{\prime}\right)
r^{n-1}dr\right)  ^{\frac{1}{s}}\nonumber \\
& =C\left \Vert \Omega \right \Vert _{L_{s}\left(  S^{n-1}\right)  }\left \vert
B\left(  x_{0},2t\right)  \right \vert ^{\frac{1}{s}}.\label{e311}%
\end{align}

Thus, by (\ref{e311}), it follows that:%
\[
\left \vert T_{\Omega,\alpha}f_{2}\left(  x\right)  \right \vert \lesssim
\int \limits_{2r}^{\infty}\left \Vert f\right \Vert _{L_{p}\left(  B\left(
x_{0},t\right)  \right)  }\frac{dt}{t^{\frac{n}{q}+1}}.
\]

Moreover, for all $p\in \left[  1,\infty \right)  $ the inequality%
\begin{equation}
\left \Vert T_{\Omega,\alpha}f_{2}\right \Vert _{L_{q}\left(  B\right)
}\lesssim r^{\frac{n}{q}}\int \limits_{2r}^{\infty}\left \Vert f\right \Vert
_{L_{p}\left(  B\left(  x_{0},t\right)  \right)  }\frac{dt}{t^{\frac{n}{q}+1}%
}\label{e312}%
\end{equation}

is valid. Thus, we obtain%
\[
\left \Vert T_{\Omega,\alpha}f\right \Vert _{L_{q}\left(  B\right)  }%
\lesssim \left \Vert f\right \Vert _{L_{p}\left(  2B\right)  }+r^{\frac{n}{q}%
}\int \limits_{2r}^{\infty}\left \Vert f\right \Vert _{L_{p}\left(  B\left(
x_{0},t\right)  \right)  }\frac{dt}{t^{\frac{n}{q}+1}}.
\]

On the other hand, we have%
\begin{align}
\left \Vert f\right \Vert _{L_{p}\left(  2B\right)  }  & \approx r^{\frac{n}{q}%
}\left \Vert f\right \Vert _{L_{p}\left(  2B\right)  }\int \limits_{2r}^{\infty
}\frac{dt}{t^{\frac{n}{q}+1}}\nonumber \\
& \leq r^{\frac{n}{q}}\int \limits_{2r}^{\infty}\left \Vert f\right \Vert
_{L_{p}\left(  B\left(  x_{0},t\right)  \right)  }\frac{dt}{t^{\frac{n}{q}+1}%
}.\label{e313}%
\end{align}

By combining the above inequalities, we obtain%
\[
\left \Vert T_{\Omega,\alpha}f\right \Vert _{L_{q}\left(  B\right)  }\lesssim
r^{\frac{n}{q}}\int \limits_{2r}^{\infty}\left \Vert f\right \Vert _{L_{p}\left(
B\left(  x_{0},t\right)  \right)  }\frac{dt}{t^{\frac{n}{q}+1}}.
\]

Let $1<q<s$. Similarly to (\ref{e311}), when $y\in B\left(  x_{0},t\right)  $,
it is true that%
\begin{equation}
\left(  \int \limits_{B\left(  x_{0},r\right)  }\left \vert \Omega \left(
x-y\right)  \right \vert ^{s}dy\right)  ^{\frac{1}{s}}\leq C\left \Vert
\Omega \right \Vert _{L_{s}\left(  S^{n-1}\right)  }\left \vert B\left(
x_{0},\frac{3}{2}t\right)  \right \vert ^{\frac{1}{s}}.\label{314}%
\end{equation}

By the Fubini's theorem, the Minkowski inequality and (\ref{314}) , we get%
\begin{align*}
\left \Vert T_{\Omega,\alpha}f_{2}\right \Vert _{L_{q}\left(  B\right)  }  &
\leq \left(  \int \limits_{B}\left \vert \int \limits_{2r}^{\infty}\int
\limits_{B\left(  x_{0},t\right)  }\left \vert f\left(  y\right)  \right \vert
\left \vert \Omega \left(  x-y\right)  \right \vert dy\frac{dt}{t^{n+1-\alpha}%
}\right \vert ^{q}dx\right)  ^{\frac{1}{q}}\\
& \leq \int \limits_{2r}^{\infty}\int \limits_{B\left(  x_{0},t\right)
}\left \vert f\left(  y\right)  \right \vert \left \Vert \Omega \left(
\cdot-y\right)  \right \Vert _{L_{q}\left(  B\right)  }dy\frac{dt}%
{t^{n+1-\alpha}}\\
& \leq \left \vert B\left(  x_{0},r\right)  \right \vert ^{\frac{1}{q}-\frac
{1}{s}}\int \limits_{2r}^{\infty}\int \limits_{B\left(  x_{0},t\right)
}\left \vert f\left(  y\right)  \right \vert \left \Vert \Omega \left(
\cdot-y\right)  \right \Vert _{L_{s}\left(  B\right)  }dy\frac{dt}%
{t^{n+1-\alpha}}\\
& \lesssim r^{\frac{n}{q}-\frac{n}{s}}\int \limits_{2r}^{\infty}\left \Vert
f\right \Vert _{L_{1}\left(  B\left(  x_{0},t\right)  \right)  }\left \vert
B\left(  x_{0},\frac{3}{2}t\right)  \right \vert ^{\frac{1}{s}}\frac
{dt}{t^{n+1-\alpha}}\\
& \lesssim r^{\frac{n}{q}-\frac{n}{s}}\int \limits_{2r}^{\infty}\left \Vert
f\right \Vert _{L_{p}\left(  B\left(  x_{0},t\right)  \right)  }t^{\frac{n}%
{s}-\frac{n}{q}-1}dt.
\end{align*}

Let $p=1<q<s\leq \infty$. From the weak $\left(  1,q\right)  $ boundedness of
$T_{\Omega,\alpha}$ and (\ref{e313}) it follows that:%
\begin{align}
\left \Vert T_{\Omega,\alpha}f_{1}\right \Vert _{WL_{q}\left(  B\right)  }  &
\leq \left \Vert T_{\Omega,\alpha}f_{1}\right \Vert _{WL_{q}\left(
\mathbb{R}
^{n}\right)  }\lesssim \left \Vert f_{1}\right \Vert _{L_{1}\left(
\mathbb{R}
^{n}\right)  }\nonumber \\
& =\left \Vert f\right \Vert _{L_{1}\left(  2B\right)  }\lesssim r^{\frac{n}{q}%
}\int \limits_{2r}^{\infty}\left \Vert f\right \Vert _{L_{1}\left(  B\left(
x_{0},t\right)  \right)  }\frac{dt}{t^{\frac{n}{q}+1}}.\label{315}%
\end{align}

Then from (\ref{e312}) and (\ref{315}) we get the inequality (\ref{e38}),
which completes the proof.
\end{proof}

In the following theorem (our main result), we get the boundedness of the
operator $T_{\Omega,\alpha}$ on the generalized local Morrey spaces
$LM_{p,\varphi}^{\left \{  x_{0}\right \}  }$.

\begin{theorem}
\label{teo9}Suppose that $x_{0}\in{\mathbb{R}^{n}}$, $\Omega \in L_{s}%
(S^{n-1})$, $1<s\leq \infty$, is homogeneous of degree zero. Let $0<\alpha<n$,
$1\leq p<\frac{n}{\alpha}$, $\frac{1}{q}=\frac{1}{p}-\frac{\alpha}{n}$. Let
$T_{\Omega,\alpha}$ be a sublinear operator satisfying condition (\ref{e1}),
bounded from $L_{p}({\mathbb{R}^{n}})$ to $L_{q}({\mathbb{R}^{n}})$ for $p>1$,
and bounded from $L_{1}({\mathbb{R}^{n}})$ to $WL_{q}({\mathbb{R}^{n}})$ for
$p=1$. Let also, for $s^{\prime}\leq p$, $p\neq1$, the pair $(\varphi
_{1},\varphi_{2})$ satisfies the condition%
\begin{equation}
\int \limits_{r}^{\infty}\frac{\operatorname*{essinf}\limits_{t<\tau<\infty
}\varphi_{1}(x_{0},\tau)\tau^{\frac{n}{p}}}{t^{\frac{n}{q}+1}}dt\leq
C\varphi_{2}(x_{0},r),\label{316}%
\end{equation}
and for $q<s$ the pair $(\varphi_{1},\varphi_{2})$ satisfies the condition%
\begin{equation}
\int \limits_{r}^{\infty}\frac{\operatorname*{essinf}\limits_{t<\tau<\infty
}\varphi_{1}(x_{0},\tau)\tau^{\frac{n}{p}}}{t^{\frac{n}{q}-\frac{n}{s}+1}%
}dt\leq C\, \varphi_{2}(x_{0},r)r^{\frac{n}{s}},\label{317}%
\end{equation}
where $C$ does not depend on $r$.

Then the operator $T_{\Omega,\alpha}$ is bounded from $LM_{p,\varphi_{1}%
}^{\{x_{0}\}}$ to $LM_{q,\varphi_{2}}^{\{x_{0}\}}$ for $p>1$ and from
$LM_{1,\varphi_{1}}^{\{x_{0}\}}$ to $WLM_{q,\varphi_{2}}^{\{x_{0}\}}$ for
$p=1$. Moreover, we have for $p>1$%
\begin{equation}
\left \Vert T_{\Omega,\alpha}f\right \Vert _{LM_{q,\varphi_{2}}^{\{x_{0}\}}%
}\lesssim \left \Vert f\right \Vert _{LM_{p,\varphi_{1}}^{\{x_{0}\}}},\label{3-1}%
\end{equation}
and for $p=1$%
\begin{equation}
\left \Vert T_{\Omega,\alpha}f\right \Vert _{WLM_{q,\varphi_{2}}^{\{x_{0}\}}%
}\lesssim \left \Vert f\right \Vert _{LM_{1,\varphi_{1}}^{\{x_{0}\}}}.\label{3-2}%
\end{equation}

\end{theorem}

\begin{proof}
Let $1<p<\infty$ and $s^{\prime}\leq p$. By Lemma \ref{lemma2} and Theorem
\ref{teo5} with $v_{2}\left(  r\right)  =\varphi_{2}\left(  x_{0},r\right)
^{-1}$, $v_{1}=\varphi_{1}\left(  x_{0},r\right)  ^{-1}r^{-\frac{n}{p}}$,
$w\left(  r\right)  =r^{-\frac{n}{q}-1}$ and $g\left(  r\right)  =\left \Vert
f\right \Vert _{L_{p}\left(  B\left(  x_{0},r\right)  \right)  }$, we have%
\begin{align*}
\left \Vert T_{\Omega,\alpha}f\right \Vert _{LM_{q,\varphi_{2}}^{\{x_{0}\}}}  &
\lesssim \sup_{r>0}\varphi_{2}\left(  x_{0},r\right)  ^{-1}\int \limits_{r}%
^{\infty}\left \Vert f\right \Vert _{L_{p}\left(  B\left(  x_{0},t\right)
\right)  }\frac{dt}{t^{\frac{n}{q}+1}}\\
& \lesssim \sup_{r>0}\varphi_{1}\left(  x_{0},r\right)  ^{-1}r^{-\frac{n}{p}%
}\left \Vert f\right \Vert _{L_{p}\left(  B\left(  x_{0},r\right)  \right)
}=\left \Vert f\right \Vert _{LM_{p,\varphi_{1}}^{\{x_{0}\}}},
\end{align*}
where condition (\ref{3*}) is equivalent to (\ref{316}), then we obtain
(\ref{3-1}).

Let $1\leq q<s$. By Lemma \ref{lemma2} and Theorem \ref{teo5} with
$v_{2}\left(  r\right)  =\varphi_{2}\left(  x_{0},r\right)  ^{-1}$,
$v_{1}=\varphi_{1}\left(  x_{0},r\right)  ^{-1}r^{-\frac{n}{p}+\frac{n}{s}}$,
$w\left(  r\right)  =r^{-\frac{n}{q}+\frac{n}{s}-1}$ and $g\left(  r\right)
=\left \Vert f\right \Vert _{L_{p}\left(  B\left(  x_{0},r\right)  \right)  }$,
we have%
\begin{align*}
\left \Vert T_{\Omega,\alpha}f\right \Vert _{LM_{q,\varphi_{2}}^{\{x_{0}\}}}  &
\lesssim \sup_{r>0}\varphi_{2}\left(  x_{0},r\right)  ^{-1}r^{-\frac{n}{s}}%
\int \limits_{r}^{\infty}\left \Vert f\right \Vert _{L_{p}\left(  B\left(
x_{0},t\right)  \right)  }\frac{dt}{t^{\frac{n}{q}-\frac{n}{s}+1}}\\
& \lesssim \sup_{r>0}\varphi_{1}\left(  x_{0},r\right)  ^{-1}r^{-\frac{n}{p}%
}\left \Vert f\right \Vert _{L_{p}\left(  B\left(  x_{0},r\right)  \right)
}=\left \Vert f\right \Vert _{LM_{p,\varphi_{1}}^{\{x_{0}\}}},
\end{align*}
where condition (\ref{3*}) is equivalent to (\ref{317}). Thus, we obtain
(\ref{3-1}).

Also, for $p=1$ we have%
\begin{align*}
\left \Vert T_{\Omega,\alpha}f\right \Vert _{WLM_{q,\varphi_{2}}^{\{x_{0}\}}}  &
\lesssim \sup_{r>0}\varphi_{2}\left(  x_{0},r\right)  ^{-1}\int \limits_{r}%
^{\infty}\left \Vert f\right \Vert _{L_{1}\left(  B\left(  x_{0},t\right)
\right)  }\frac{dt}{t^{\frac{n}{q}+1}}\\
& \lesssim \sup_{r>0}\varphi_{1}\left(  x_{0},r\right)  ^{-1}r^{-n}\left \Vert
f\right \Vert _{L_{1}\left(  B\left(  x_{0},r\right)  \right)  }=\left \Vert
f\right \Vert _{LM_{1,\varphi_{1}}^{\{x_{0}\}}}.
\end{align*}
Hence, the proof is completed.
\end{proof}

In the case of $q=\infty$ by Theorem \ref{teo9}, we get

\begin{corollary}
\label{corollary 3}Let $x_{0}\in{\mathbb{R}^{n}}$, $1\leq p<\infty$,
$0<\alpha<\frac{n}{p}$, $\frac{1}{q}=\frac{1}{p}-\frac{\alpha}{n}$ and the
pair $(\varphi_{1},\varphi_{2})$ satisfies condition (\ref{316}). Then the
operators $M_{\alpha}$ and $\overline{T}_{\alpha}$ are bounded from
$LM_{p,\varphi_{1}}^{\{x_{0}\}}$ to $LM_{q,\varphi_{2}}^{\{x_{0}\}}$ for $p>1$
and from $LM_{1,\varphi_{1}}^{\{x_{0}\}}$ to $WLM_{q,\varphi_{2}}^{\{x_{0}\}}$
for $p=1$.
\end{corollary}

\begin{corollary}
Suppose that $x_{0}\in{\mathbb{R}^{n}}$, $\Omega \in L_{s}(S^{n-1})$,
$1<s\leq \infty$, is homogeneous of degree zero. Let $0<\alpha<n$, $1\leq
p<\frac{n}{\alpha}$ and $\frac{1}{q}=\frac{1}{p}-\frac{\alpha}{n}$. Let also
for $s^{\prime}\leq p$ the pair $\left(  \varphi_{1},\varphi_{2}\right)  $
satisfies condition (\ref{316}) and for $q<s$ the pair $\left(  \varphi
_{1},\varphi_{2}\right)  $ satisfies condition (\ref{317}). Then the operators
$M_{\Omega,\alpha}$ and $\overline{T}_{\Omega,\alpha}$ are bounded from
$LM_{p,\varphi_{1}}^{\{x_{0}\}}$ to $LM_{q,\varphi_{2}}^{\{x_{0}\}}$ for $p>1$
and from $LM_{1,\varphi_{1}}^{\{x_{0}\}}$ to $WLM_{q,\varphi_{2}}^{\{x_{0}\}}$
for $p=1$.
\end{corollary}

Now using above results, we get the boundedness of the operator $T_{\Omega
,\alpha}$ on the generalized vanishing local Morrey spaces $VLM_{p,\varphi
}^{\left \{  x_{0}\right \}  }$.

\begin{theorem}
\label{teo10}(Our main result) Let $x_{0}\in{\mathbb{R}^{n}}$, $\Omega \in
L_{s}(S^{n-1})$, $1<s\leq \infty$, be homogeneous of degree zero. Let
$0<\alpha<n$, $1\leq p<\frac{n}{\alpha}$ and $\frac{1}{q}=\frac{1}{p}%
-\frac{\alpha}{n}$. Let $T_{\Omega,\alpha}$ be a sublinear operator satisfying
condition (\ref{e1}), bounded on $L_{p}({\mathbb{R}^{n}})$ for $p>1$, and
bounded from $L_{1}({\mathbb{R}^{n}})$ to $WL_{1}({\mathbb{R}^{n}})$. Let for
$s^{\prime}\leq p$, $p\neq1$, the pair $(\varphi_{1},\varphi_{2})$ satisfies
conditions (\ref{2})-(\ref{3}) and
\begin{equation}
c_{\delta}:=%
{\displaystyle \int \limits_{\delta}^{\infty}}
\varphi_{1}\left(  x_{0},t\right)  \frac{t^{\frac{n}{p}}}{t^{\frac{n}{q}+1}%
}dt<\infty \label{6}%
\end{equation}
for every $\delta>0$, and
\begin{equation}
\int \limits_{r}^{\infty}\varphi_{1}\left(  x_{0},t\right)  \frac{t^{\frac
{n}{p}}}{t^{\frac{n}{q}+1}}dt\leq C_{0}\varphi_{2}(x_{0},r),\label{7}%
\end{equation}
and for $q<s$ the pair $(\varphi_{1},\varphi_{2})$ satisfies conditions
(\ref{2})-(\ref{3}) and also%
\begin{equation}
c_{\delta^{\prime}}:=%
{\displaystyle \int \limits_{\delta^{\prime}}^{\infty}}
\varphi_{1}(x_{0},t)\frac{t^{\frac{n}{p}}}{t^{\frac{n}{q}-\frac{n}{s}+1}%
}dt<\infty \label{8}%
\end{equation}
for every $\delta^{\prime}>0$, and%
\begin{equation}
\int \limits_{r}^{\infty}\varphi_{1}(x_{0},t)\frac{t^{\frac{n}{p}}}{t^{\frac
{n}{q}-\frac{n}{s}+1}}dt\leq C_{0}\varphi_{2}(x_{0},r)r^{\frac{n}{s}%
},\label{9}%
\end{equation}
where $C_{0}$ does not depend on $r>0$.

Then the operator $T_{\Omega,\alpha}$ is bounded from $VLM_{p,\varphi_{1}%
}^{\{x_{0}\}}$ to $VLM_{q,\varphi_{2}}^{\{x_{0}\}}$ for $p>1$ and from
$VLM_{1,\varphi_{1}}^{\{x_{0}\}}$ to $WVLM_{q,\varphi_{2}}^{\{x_{0}\}}$ for
$p=1$. Moreover, we have for $p>1$%
\begin{equation}
\left \Vert T_{\Omega,\alpha}f\right \Vert _{VLM_{q,\varphi_{2}}^{\{x_{0}\}}%
}\lesssim \left \Vert f\right \Vert _{VLM_{p,\varphi_{1}}^{\{x_{0}\}}},\label{10}%
\end{equation}
and for $p=1$%
\begin{equation}
\left \Vert T_{\Omega,\alpha}f\right \Vert _{WVLM_{q,\varphi_{2}}^{\{x_{0}\}}%
}\lesssim \left \Vert f\right \Vert _{VLM_{1,\varphi_{1}}^{\{x_{0}\}}}.\label{11}%
\end{equation}

\end{theorem}

\begin{proof}
The norm inequalities follow from Theorem \ref{teo9}. Thus we only have to
prove that%
\begin{equation}
\lim \limits_{r\rightarrow0}\mathfrak{M}_{p,\varphi_{1}}\left(  f;x_{0}%
,r\right)  =0\text{ implies }\lim \limits_{r\rightarrow0}\mathfrak{M}%
_{q,\varphi_{2}}\left(  T_{\Omega,\alpha}f;x_{0},r\right)  =0\label{12}%
\end{equation}
and%
\begin{equation}
\lim \limits_{r\rightarrow0}\mathfrak{M}_{p,\varphi_{1}}\left(  f;x_{0}%
,r\right)  =0\text{ implies }\lim \limits_{r\rightarrow0}\mathfrak{M}%
_{q,\varphi_{2}}^{W}\left(  T_{\Omega,\alpha}f;x_{0},r\right)  =0.\label{0}%
\end{equation}

To show that $\frac{r^{-\frac{n}{q}}\left \Vert T_{\Omega,\alpha}f\right \Vert
_{L_{q}\left(  B\left(  x_{0},r\right)  \right)  }}{\varphi_{2}(x_{0}%
,r)}<\epsilon$ for small $r$, we split the right-hand side of (\ref{100}):%
\begin{equation}
\frac{r^{-\frac{n}{q}}\left \Vert T_{\Omega,\alpha}f\right \Vert _{L_{q}\left(
B\left(  x_{0},r\right)  \right)  }}{\varphi_{2}(x_{0},r)}\leq C\left[
I_{\delta_{0}}\left(  x_{0},r\right)  +J_{\delta_{0}}\left(  x_{0},r\right)
\right]  ,\label{13}%
\end{equation}
where $\delta_{0}>0$ (we may take $\delta_{0}<1$), and
\[
I_{\delta_{0}}\left(  x_{0},r\right)  :=\frac{1}{\varphi_{2}(x_{0},r)}%
{\displaystyle \int \limits_{r}^{\delta_{0}}}
t^{-\frac{n}{q}-1}\left \Vert f\right \Vert _{L_{p}\left(  B\left(
x_{0},t\right)  \right)  }dt,
\]
and%
\[
J_{\delta_{0}}\left(  x_{0},r\right)  :=\frac{1}{\varphi_{2}(x_{0},r)}%
{\displaystyle \int \limits_{\delta_{0}}^{\infty}}
t^{-\frac{n}{q}-1}\left \Vert f\right \Vert _{L_{p}\left(  B\left(
x_{0},t\right)  \right)  }dt,
\]
and $r<\delta_{0}$. Now we use the fact that $f\in VLM_{p,\varphi_{1}%
}^{\{x_{0}\}}$ and we choose any fixed $\delta_{0}>0$ such that%
\[
\frac{t^{-\frac{n}{p}}\left \Vert f\right \Vert _{L_{p}\left(  B\left(
x_{0},t\right)  \right)  }}{\varphi_{1}(x_{0},t)}<\frac{\epsilon}{2CC_{0}%
},\qquad t\leq \delta_{0},
\]
where $C$ and $C_{0}$ are constants from (\ref{7}) and (\ref{13}). This allows
to estimate the first term uniformly in $r\in \left(  0,\delta_{0}\right)  :$%
\[
CI_{\delta_{0}}\left(  x_{0},r\right)  <\frac{\epsilon}{2},\qquad
0<r<\delta_{0}.
\]

The estimation of the second term may be obtained by choosing $r$ sufficiently
small. Indeed, we have%
\[
J_{\delta_{0}}\left(  x_{0},r\right)  \leq c_{\delta_{0}}\frac{\left \Vert
f\right \Vert _{LM_{p,\varphi_{1}}^{\{x_{0}\}}}}{\varphi_{2}\left(
x_{0},r\right)  },
\]
where $c_{\delta_{0}}$ is the constant from (\ref{6}) with $\delta=\delta_{0}%
$. Then, by (\ref{2}) it suffices to choose $r$ small enough such that
\[
\frac{1}{\varphi_{2}(x_{0},r)}\leq \frac{\epsilon}{2c_{\delta_{0}}\left \Vert
f\right \Vert _{LM_{p,\varphi_{1}}^{\{x_{0}\}}}},
\]
which completes the proof of (\ref{12}).

The proof of (\ref{0}) is similar to the proof of (\ref{12}). For the case of
$q<s$, we can also use the same method, so we omit the details.
\end{proof}

\begin{remark}
Conditions (\ref{6}) and (\ref{8}) are not needed in the case when
$\varphi(x,r)$ does not depend on $x$, since (\ref{6}) follows from (\ref{7})
and similarly, (\ref{8}) follows from (\ref{9}) in this case.
\end{remark}

\begin{corollary}
Let $x_{0}\in{\mathbb{R}^{n}}$, $\Omega \in L_{s}(S^{n-1})$, $1<s\leq \infty$,
be homogeneous of degree zero. Let $0<\alpha<n$, $1\leq p<\frac{n}{\alpha}$
and $\frac{1}{q}=\frac{1}{p}-\frac{\alpha}{n}$. Let also for $s^{\prime}\leq
p$, $p\neq1$, the pair $(\varphi_{1},\varphi_{2})$ satisfies conditions
(\ref{2})-(\ref{3}) and (\ref{6})-(\ref{7}) and for $q<s $ the pair $\left(
\varphi_{1},\varphi_{2}\right)  $ satisfies conditions (\ref{2})-(\ref{3}) and
(\ref{8})-(\ref{9}). Then the operators $M_{\Omega,\alpha}$ and $\overline
{T}_{\Omega,\alpha}$ are bounded from $VLM_{p,\varphi_{1}}^{\{x_{0}\}}$ to
$VLM_{q,\varphi_{2}}^{\{x_{0}\}}$ for $p>1$ and from $VLM_{1,\varphi_{1}%
}^{\{x_{0}\}}$ to $WVLM_{q,\varphi_{2}}^{\{x_{0}\}}$ for $p=1$.
\end{corollary}

In the case of $q=\infty$ by Theorem \ref{teo10}, we get

\begin{corollary}
Let $x_{0}\in{\mathbb{R}^{n}}$,$1\leq p<\infty$ and the pair $(\varphi
_{1},\varphi_{2})$ satisfies conditions (\ref{2})-(\ref{3}) and (\ref{6}%
)-(\ref{7}). Then the operators $M_{\alpha}$ and $\overline{T}_{\alpha}$ are
bounded from $VLM_{p,\varphi_{1}}^{\{x_{0}\}}$ to $VLM_{q,\varphi_{2}%
}^{\{x_{0}\}}$ for $p>1$ and from $VLM_{1,\varphi_{1}}^{\{x_{0}\}}$ to
$WVLM_{q,\varphi_{2}}^{\{x_{0}\}}$ for $p=1$.
\end{corollary}

\section{Commutators of the linear operators with rough kernel $T_{\Omega
,\alpha}$ on the spaces $LM_{p,\varphi}^{\left \{  x_{0}\right \}  }$ and
$VLM_{p,\varphi}^{\left \{  x_{0}\right \}  }$}

In this section, we will first prove the boundedness of the operator
$T_{\Omega,b,\alpha}$ with $b\in LC_{p_{2},\lambda}^{\left \{  x_{0}\right \}
}$ on the generalized local Morrey spaces $LM_{p,\varphi}^{\left \{
x_{0}\right \}  }$ by using the following weighted Hardy operator%

\[
H_{\omega}^{\ast}g(r):=\int \limits_{r}^{\infty}\left(  1+\ln \frac{t}%
{r}\right)  g(t)\omega(t)dt,\text{ \  \  \ }r\in \left(  0,\infty \right)  ,
\]
where $\omega$ is a weight function. Then, we will also obtain the boundedness
of $T_{\Omega,b,\alpha}$ with $b\in LC_{p_{2},\lambda}^{\left \{
x_{0}\right \}  }$ on generalized vanishing local Morrey spaces $VLM_{p,\varphi
}^{\left \{  x_{0}\right \}  }$.

Let $T$ be a linear operator. For a locally integrable function $b$ on
${\mathbb{R}^{n}}$, we define the commutator $[b,T]$ by
\[
\lbrack b,T]f(x)=b(x)\,Tf(x)-T(bf)(x)
\]
for any suitable function $f$. Let $\overline{T}$ be a C--Z operator. A well
known result of Coifman et al. \cite{CRW} states that when $K\left(  x\right)
=\frac{\Omega \left(  x^{\prime}\right)  }{\left \vert x\right \vert ^{n}}$ and
$\Omega$ is smooth, the commutator $[b,\overline{T}]f=b\, \overline
{T}f-\overline{T}(bf)$ is bounded on $L_{p}({\mathbb{R}^{n}})$, $1<p<\infty$,
if and only if $b\in BMO({\mathbb{R}^{n}})$.

Since $BMO({\mathbb{R}^{n}})\subset \bigcap \limits_{p>1}LC_{p}^{\left \{
x_{0}\right \}  }({\mathbb{R}^{n}})$, if we only assume $b\in LC_{p}^{\left \{
x_{0}\right \}  }({\mathbb{R}^{n}})$, or more generally $b\in LC_{p,\lambda
}^{\left \{  x_{0}\right \}  }({\mathbb{R}^{n}})$, then $[b,\overline{T}]$ may
not be a bounded operator on $L_{p}({\mathbb{R}^{n}})$, $1<p<\infty$. However,
it has some boundedness properties on other spaces. As a matter of fact,
Grafakos et al. \cite{GraLiYang} have considered the commutator with $b\in
LC_{p}^{\left \{  x_{0}\right \}  }({\mathbb{R}^{n}})$ on Herz spaces for the
first time. Morever, in \cite{FuLinLu} and \cite{TaoShi}, they have considered
the commutators with $b\in LC_{p,\lambda}^{\left \{  x_{0}\right \}
}({\mathbb{R}^{n}})$. The commutator of C--Z operators plays an important role
in studying the regularity of solutions of elliptic partial differential
equations of second order (see, for example, \cite{ChFraL1, ChFraL2,
FazRag2}). The boundedness of the commutator has been generalized to other
contexts and important applications to some non-linear PDEs have been given by
Coifman et al. \cite{CLMS}. On the other hand, For $b\in L_{1}^{loc}%
({\mathbb{R}^{n}})$, the commutator $[b,\overline{T}_{\alpha}]$ of fractional
integral operator (also known as the Riesz potential) is defined by
\[
\lbrack b,\overline{T}_{\alpha}]f(x)=b(x)\overline{T}_{\alpha}f(x)-\overline
{T}_{\alpha}(bf)(x)=\int \limits_{{\mathbb{R}^{n}}}\frac{b(x)-b(y)}%
{|x-y|^{n-\alpha}}f(y)dy\qquad0<\alpha<n
\]
for any suitable function $f$.

The function $b$ is also called the symbol function of $[b,\overline
{T}_{\alpha}]$. The characterization of $\left(  L_{p},L_{q}\right)
$-boundedness of the commutator $[b,\overline{T}_{\alpha}]$ of fractional
integral operator has been given by Chanillo \cite{Chanillo}. A well known
result of Chanillo \cite{Chanillo} states that the commutator $[b,\overline
{T}_{\alpha}]$ is bounded from $L_{p}({\mathbb{R}^{n}})$ to $L_{q}%
({\mathbb{R}^{n}})$, $1<p<q<\infty$, $\frac{1}{p}-\frac{1}{q}=\frac{\alpha}%
{n}$ if and only if $b\in BMO({\mathbb{R}^{n}})$. There are two major reasons
for considering the problem of commutators. The first one is that the
boundedness of commutators can produce some characterizations of function
spaces (see \cite{BGGS, Chanillo, GulJMS2013, Gurbuz, Gurbuz2, Gurbuz3,
Janson, Palus, Shi}). The other one is that the theory of commutators plays an
important role in the study of the regularity of solutions to elliptic and
parabolic PDEs of the second order (see \cite{ChFraL1, ChFraL2, FazRag2,
FazPalRag, Softova}).

The definition of local Campanato space $LC_{p,\lambda}^{\left \{
x_{0}\right \}  }$ as follows.

\begin{definition}
\cite{BGGS, GulJMS2013, Gurbuz} Let $1\leq p<\infty$ and $0\leq \lambda
<\frac{1}{n}$. A function $f\in L_{p}^{loc}\left(  {\mathbb{R}^{n}}\right)  $
is said to belong to the $LC_{p,\lambda}^{\left \{  x_{0}\right \}  }\left(
{\mathbb{R}^{n}}\right)  $ (local Campanato space), if%
\begin{equation}
\left \Vert f\right \Vert _{LC_{p,\lambda}^{\left \{  x_{0}\right \}  }}%
=\sup_{r>0}\left(  \frac{1}{\left \vert B\left(  x_{0},r\right)  \right \vert
^{1+\lambda p}}\int \limits_{B\left(  x_{0},r\right)  }\left \vert f\left(
y\right)  -f_{B\left(  x_{0},r\right)  }\right \vert ^{p}dy\right)  ^{\frac
{1}{p}}<\infty,\label{e51}%
\end{equation}

where%
\[
f_{B\left(  x_{0},r\right)  }=\frac{1}{\left \vert B\left(  x_{0},r\right)
\right \vert }\int \limits_{B\left(  x_{0},r\right)  }f\left(  y\right)  dy.
\]

Define%
\[
LC_{p,\lambda}^{\left \{  x_{0}\right \}  }\left(  {\mathbb{R}^{n}}\right)
=\left \{  f\in L_{p}^{loc}\left(  {\mathbb{R}^{n}}\right)  :\left \Vert
f\right \Vert _{LC_{p,\lambda}^{\left \{  x_{0}\right \}  }}<\infty \right \}  .
\]

\end{definition}

\begin{remark}
If two functions which differ by a constant are regarded as a function in the
space $LC_{p,\lambda}^{\left \{  x_{0}\right \}  }\left(  {\mathbb{R}^{n}%
}\right)  $, then $LC_{p,\lambda}^{\left \{  x_{0}\right \}  }\left(
{\mathbb{R}^{n}}\right)  $ becomes a Banach space. The space $LC_{p,\lambda
}^{\left \{  x_{0}\right \}  }\left(  {\mathbb{R}^{n}}\right)  $ when
$\lambda=0$ is just the $LC_{p}^{\left \{  x_{0}\right \}  }({\mathbb{R}^{n}})$.
Apparently, (\ref{e51}) is equivalent to the following condition:%
\[
\sup_{r>0}\inf_{c\in%
\mathbb{C}
}\left(  \frac{1}{\left \vert B\left(  x_{0},r\right)  \right \vert ^{1+\lambda
p}}\int \limits_{B\left(  x_{0},r\right)  }\left \vert f\left(  y\right)
-c\right \vert ^{p}dy\right)  ^{\frac{1}{p}}<\infty.
\]

\end{remark}

In \cite{LuYang1}, Lu and Yang have introduced the central BMO space
$CBMO_{p}({\mathbb{R}^{n}})=LC_{p,0}^{\{0\}}({\mathbb{R}^{n}})$. Also the
space $CBMO^{\{x_{0}\}}({\mathbb{R}^{n}})=LC_{1,0}^{\{x_{0}\}}({\mathbb{R}%
^{n}})$ has been considered in other denotes in \cite{Rzaev}. The space
$LC_{p}^{\left \{  x_{0}\right \}  }({\mathbb{R}^{n}})$ can be regarded as a
local version of $BMO({\mathbb{R}^{n}})$, the space of bounded mean
oscillation, at the origin. But, they have quite different properties. The
classical John-Nirenberg inequality shows that functions in $BMO({\mathbb{R}%
^{n}})$ are locally exponentially integrable. This implies that, for any
$1\leq p<\infty$, the functions in $BMO({\mathbb{R}^{n}})$ can be described by
means of the condition:%
\[
\sup_{B\subset{\mathbb{R}^{n}}}\left(  \frac{1}{|B|}\int_{B}|f(y)-f_{B}%
|^{p}dy\right)  ^{\frac{1}{p}}<\infty,
\]
where $B$ denotes an arbitrary ball in ${\mathbb{R}^{n}}$. However, the space
$LC_{p}^{\left \{  x_{0}\right \}  }({\mathbb{R}^{n}})$ depends on $p$. If
$p_{1}<p_{2}$, then $LC_{p_{2}}^{\left \{  x_{0}\right \}  }({\mathbb{R}^{n}%
})\subsetneqq LC_{p_{1}}^{\left \{  x_{0}\right \}  }({\mathbb{R}^{n}})$.
Therefore, there is no analogy of the famous John-Nirenberg inequality of
$BMO({\mathbb{R}^{n}})$ for the space $LC_{p}^{\left \{  x_{0}\right \}
}({\mathbb{R}^{n}})$. One can imagine that the behavior of $LC_{p}^{\left \{
x_{0}\right \}  }({\mathbb{R}^{n}})$ may be quite different from that of
$BMO({\mathbb{R}^{n}})$ (see \cite{LuWu}).

\begin{theorem}
\label{teo13}\cite{BGGS, GulJMS2013, Gurbuz, Gurbuz2, Gurbuz3} Let $v_{1}$,
$v_{2}$ and $\omega$ be weigths on $(0,\infty)$ and $v_{1}\left(  t\right)  $
be bounded outside a neighbourhood of the origin. The inequality
\begin{equation}
\operatorname*{esssup}\limits_{r>0}v_{2}(r)H_{\omega}^{\ast}g(r)\leq
C\operatorname*{esssup}\limits_{r>0}v_{1}(r)g(r)\label{e52}%
\end{equation}
holds for some $C>0$ for all non-negative and non-decreasing functions $g$ on
$(0,\infty)$ if and only if
\begin{equation}
B:=\sup \limits_{r>0}v_{2}(r)\int \limits_{r}^{\infty}\left(  1+\ln \frac{t}%
{r}\right)  \frac{\omega(t)dt}{\operatorname*{esssup}\limits_{t<s<\infty}%
v_{1}(s)}<\infty.\label{e53}%
\end{equation}
Moreover, the value $C=B$ is the best constant for (\ref{e52}).
\end{theorem}

\begin{remark}
In (\ref{e52}) and (\ref{e53}) it is assumed that $\frac{1}{\infty}=0$ and
$0.\infty=0$.
\end{remark}

\begin{lemma}
\label{Lemma 4}Let $b$ be function in $LC_{p,\lambda}^{\left \{  x_{0}\right \}
}\left(
\mathbb{R}
^{n}\right)  $, $1\leq p<\infty$, $0\leq \lambda<\frac{1}{n}$ and $r_{1}$,
$r_{2}>0$. Then%
\begin{equation}
\left(  \frac{1}{\left \vert B\left(  x_{0},r_{1}\right)  \right \vert
^{1+\lambda p}}%
{\displaystyle \int \limits_{B\left(  x_{0},r_{1}\right)  }}
\left \vert b\left(  y\right)  -b_{B\left(  x_{0},r_{2}\right)  }\right \vert
^{p}dy\right)  ^{\frac{1}{p}}\leq C\left(  1+\ln \frac{r_{1}}{r_{2}}\right)
\left \Vert b\right \Vert _{LC_{p,\lambda}^{\left \{  x_{0}\right \}  }},\label{a}%
\end{equation}
where $C>0$ is independent of $b$, $r_{1}$ and $r_{2}$.

From this inequality $\left(  \text{\ref{a}}\right)  $, we have%
\begin{equation}
\left \vert b_{B\left(  x_{0},r_{1}\right)  }-b_{B\left(  x_{0},r_{2}\right)
}\right \vert \leq C\left(  1+\ln \frac{r_{1}}{r_{2}}\right)  \left \vert
B\left(  x_{0},r_{1}\right)  \right \vert ^{\lambda}\left \Vert b\right \Vert
_{LC_{p,\lambda}^{\left \{  x_{0}\right \}  }},\label{b}%
\end{equation}

and it is easy to see that%
\begin{equation}
\left \Vert b-\left(  b\right)  _{B}\right \Vert _{L_{p}\left(  B\right)  }\leq
C\left(  1+\ln \frac{r_{1}}{r_{2}}\right)  r^{\frac{n}{p}+n\lambda}\left \Vert
b\right \Vert _{LC_{p,\lambda}^{\left \{  x_{0}\right \}  }}.\label{c}%
\end{equation}

\end{lemma}

As in the proof of Theorem \ref{teo9}, it suffices to prove the following
Lemma \ref{Lemma 5}.

\begin{lemma}
\label{Lemma 5} Let $x_{0}\in{\mathbb{R}^{n}}$, $\Omega \in L_{s}(S^{n-1})$,
$1<s\leq \infty$, be homogeneous of degree zero. Let $T_{\Omega,\alpha}$ be a
linear operator satisfying condition (\ref{e1}), bounded from $L_{p}\left(
{\mathbb{R}^{n}}\right)  $ to $L_{q}\left(  {\mathbb{R}^{n}}\right)  $. Let
also $0<\alpha<n$, $1<p<\frac{n}{\alpha}$, $b\in LC_{p_{2},\lambda}^{\left \{
x_{0}\right \}  }({\mathbb{R}^{n}})$, $0\leq \lambda<\frac{1}{n}$, $\frac{1}%
{p}=\frac{1}{p_{1}}+\frac{1}{p_{2}}$, $\frac{1}{q}=\frac{1}{p}-\frac{\alpha
}{n}$, $\frac{1}{q_{1}}=\frac{1}{p_{1}}-\frac{\alpha}{n}$.Then, for
$s^{\prime}\leq p$ the inequality
\begin{equation}
\Vert T_{\Omega,b,\alpha}f\Vert_{L_{q}(B(x_{0},r))}\lesssim \Vert
b\Vert_{LC_{p_{2},\lambda}^{\{x_{0}\}}}\,r^{\frac{n}{q}}%
{\displaystyle \int \limits_{2r}^{\infty}}
\left(  1+\ln \frac{t}{r}\right)  t^{n\lambda-\frac{n}{q_{1}}-1}\Vert
f\Vert_{L_{p_{1}}(B(x_{0},t))}dt\label{200}%
\end{equation}
holds for any ball $B(x_{0},r)$ and for all $f\in L_{p_{1}}^{loc}%
({\mathbb{R}^{n}})$.

Also, for $q_{1}<s$ the inequality%
\[
\Vert T_{\Omega,b,\alpha}f\Vert_{L_{q}(B(x_{0},r))}\lesssim \Vert
b\Vert_{LC_{p_{2},\lambda}^{\{x_{0}\}}}\,r^{\frac{n}{q}-\frac{n}{s}}%
{\displaystyle \int \limits_{2r}^{\infty}}
\left(  1+\ln \frac{t}{r}\right)  t^{n\lambda-\frac{n}{q_{1}}+\frac{n}{s}%
-1}\Vert f\Vert_{L_{p_{1}}(B(x_{0},t))}dt
\]
holds for any ball $B(x_{0},r)$ and for all $f\in L_{p_{1}}^{loc}%
({\mathbb{R}^{n}})$.
\end{lemma}

\begin{proof}
Let $1<p<\infty$, $0<\alpha<\frac{n}{p}$, $\frac{1}{p}=\frac{1}{p_{1}}%
+\frac{1}{p_{2}}$, $\frac{1}{q}=\frac{1}{p}-\frac{\alpha}{n}$ and $\frac
{1}{q_{1}}=\frac{1}{p_{1}}-\frac{\alpha}{n}$.As in the proof of Lemma
\ref{lemma2}, we represent $f$ in form (\ref{e39}) and have%
\begin{align*}
T_{\Omega,b,\alpha}f\left(  x\right)   & =\left(  b\left(  x\right)
-b_{B}\right)  T_{\Omega,\alpha}f_{1}\left(  x\right)  -T_{\Omega,\alpha
}\left(  \left(  b\left(  \cdot \right)  -b_{B}\right)  f_{1}\right)  \left(
x\right) \\
& +\left(  b\left(  x\right)  -b_{B}\right)  T_{\Omega,\alpha}f_{2}\left(
x\right)  -T_{\Omega,\alpha}\left(  \left(  b\left(  \cdot \right)
-b_{B}\right)  f_{2}\right)  \left(  x\right) \\
& \equiv J_{1}+J_{2}+J_{3}+J_{4}.
\end{align*}
Hence we get%
\[
\left \Vert T_{\Omega,b,\alpha}f\right \Vert _{L_{q}\left(  B\right)  }%
\leq \left \Vert J_{1}\right \Vert _{L_{q}\left(  B\right)  }+\left \Vert
J_{2}\right \Vert _{L_{q}\left(  B\right)  }+\left \Vert J_{3}\right \Vert
_{L_{q}\left(  B\right)  }+\left \Vert J_{4}\right \Vert _{L_{q}\left(
B\right)  }.
\]

By the H\"{o}lder's inequality, the boundedness of $T_{\Omega,\alpha}$ from
$L_{p_{1}}({\mathbb{R}^{n}})$ to $L_{q_{1}}\left(  {\mathbb{R}^{n}}\right)  $
(see Lemma \ref{Lemma1}) it follows that:%
\begin{align*}
\left \Vert J_{1}\right \Vert _{L_{q}\left(  B\right)  }  & \leq \left \Vert
\left(  b\left(  \cdot \right)  -b_{B}\right)  T_{\Omega,\alpha}f_{1}\left(
\cdot \right)  \right \Vert _{L_{q}\left(  {\mathbb{R}^{n}}\right)  }\\
& \lesssim \left \Vert \left(  b\left(  \cdot \right)  -b_{B}\right)  \right \Vert
_{L_{p_{2}}\left(  {\mathbb{R}^{n}}\right)  }\left \Vert T_{\Omega,\alpha}%
f_{1}\left(  \cdot \right)  \right \Vert _{L_{q_{1}}\left(  {\mathbb{R}^{n}%
}\right)  }\\
& \lesssim \left \Vert b\right \Vert _{LC_{p_{2},\lambda}^{\left \{
x_{0}\right \}  }}r^{\frac{n}{p_{2}}+n\lambda}\left \Vert f_{1}\right \Vert
_{L_{p_{1}}\left(  {\mathbb{R}^{n}}\right)  }\\
& =\left \Vert b\right \Vert _{LC_{p_{2},\lambda}^{\left \{  x_{0}\right \}  }%
}r^{\frac{n}{p_{2}}+\frac{n}{q_{1}}+n\lambda}\left \Vert f\right \Vert
_{L_{p_{1}}\left(  2B\right)  }\int \limits_{2r}^{\infty}t^{-1-\frac{n}{q_{1}}%
}dt\\
& \lesssim \left \Vert b\right \Vert _{LC_{p_{2},\lambda}^{\left \{
x_{0}\right \}  }}r^{\frac{n}{q}+n\lambda}\int \limits_{2r}^{\infty}\left(
1+\ln \frac{t}{r}\right)  \left \Vert f\right \Vert _{L_{p_{1}}\left(  B\left(
x_{0},t\right)  \right)  }t^{-1-\frac{n}{q_{1}}}dt.
\end{align*}

Using the the boundedness of $T_{\Omega,\alpha}$ from $L_{p}({\mathbb{R}^{n}%
})$ to $L_{q}\left(  {\mathbb{R}^{n}}\right)  $ (see Lemma \ref{Lemma1}), by
the H\"{o}lder's inequality for $J_{2}$, we have%
\begin{align*}
\left \Vert J_{2}\right \Vert _{L_{q}\left(  B\right)  } &  \leq \left \Vert
T_{\Omega,\alpha}^{P}\left(  b\left(  \cdot \right)  -b_{B}\right)
f_{1}\right \Vert _{L_{q}\left(  {\mathbb{R}^{n}}\right)  }\\
&  \lesssim \left \Vert \left(  b\left(  \cdot \right)  -b_{B}\right)
f_{1}\right \Vert _{L_{p}\left(  {\mathbb{R}^{n}}\right)  }\\
&  \lesssim \left \Vert b\left(  \cdot \right)  -b_{B}\right \Vert _{L_{p_{2}%
}\left(  {\mathbb{R}^{n}}\right)  }\left \Vert f_{1}\right \Vert _{L_{p_{1}%
}\left(  {\mathbb{R}^{n}}\right)  }\\
&  \lesssim \left \Vert b\right \Vert _{LC_{p_{2},\lambda}^{\left \{
x_{0}\right \}  }}r^{\frac{n}{p_{2}}+\frac{n}{q_{1}}+n\lambda}\left \Vert
f\right \Vert _{L_{p_{1}}\left(  2B\right)  }\int \limits_{2r}^{\infty
}t^{-1-\frac{n}{q_{1}}}dt\\
&  \lesssim \left \Vert b\right \Vert _{LC_{p_{2},\lambda}^{\left \{
x_{0}\right \}  }}r^{\frac{n}{q}+n\lambda}\int \limits_{2r}^{\infty}\left(
1+\ln \frac{t}{r}\right)  \left \Vert f\right \Vert _{L_{p_{1}}\left(  B\left(
x_{0},t\right)  \right)  }t^{-1-\frac{n}{q_{1}}}dt.
\end{align*}

For $J_{3}$, it is known that $x\in B$, $y\in \left(  2B\right)  ^{C}$, which
implies \ $\frac{1}{2}\left \vert x_{0}-y\right \vert \leq \left \vert
x-y\right \vert \leq \frac{3}{2}\left \vert x_{0}-y\right \vert $.

When $s^{\prime}\leq p_{1}$, by the Fubini's theorem, the H\"{o}lder's
inequality and (\ref{e311}), we have%
\begin{align*}
\left \vert T_{\Omega,\alpha}f_{2}\left(  x\right)  \right \vert  &  \leq
c_{0}\int \limits_{\left(  2B\right)  ^{C}}\left \vert \Omega \left(  x-y\right)
\right \vert \frac{\left \vert f\left(  y\right)  \right \vert }{\left \vert
x_{0}-y\right \vert ^{n-\alpha}}dy\\
&  \approx \int \limits_{2r}^{\infty}\int \limits_{2r<\left \vert x_{0}%
-y\right \vert <t}\left \vert \Omega \left(  x-y\right)  \right \vert \left \vert
f\left(  y\right)  \right \vert dyt^{-1-n+\alpha}dt\\
&  \lesssim \int \limits_{2r}^{\infty}\int \limits_{B\left(  x_{0},t\right)
}\left \vert \Omega \left(  x-y\right)  \right \vert \left \vert f\left(
y\right)  \right \vert dyt^{-1-n+\alpha}dt\\
&  \lesssim \int \limits_{2r}^{\infty}\left \Vert f\right \Vert _{L_{p_{1}}\left(
B\left(  x_{0},t\right)  \right)  }\left \Vert \Omega \left(  x-\cdot \right)
\right \Vert _{L_{s}\left(  B\left(  x_{0},t\right)  \right)  }\left \vert
B\left(  x_{0},t\right)  \right \vert ^{1-\frac{1}{p_{1}}-\frac{1}{s}%
}t^{-1-n+\alpha}dt\\
&  \lesssim \int \limits_{2r}^{\infty}\left \Vert f\right \Vert _{L_{p_{1}}\left(
B\left(  x_{0},t\right)  \right)  }t^{-1-\frac{n}{q_{1}}}dt.
\end{align*}
Hence, we get%
\begin{align*}
\left \Vert J_{3}\right \Vert _{L_{q}\left(  B\right)  } &  \leq \left \Vert
\left(  b\left(  \cdot \right)  -b_{B}\right)  T_{\Omega,\alpha}f_{2}\left(
\cdot \right)  \right \Vert _{L_{q}\left(  {\mathbb{R}^{n}}\right)  }\\
&  \lesssim \left \Vert \left(  b\left(  \cdot \right)  -b_{B}\right)
\right \Vert _{L_{q}\left(  {\mathbb{R}^{n}}\right)  }\int \limits_{2r}^{\infty
}\left \Vert f\right \Vert _{L_{p_{1}}\left(  B\left(  x_{0},t\right)  \right)
}t^{-1-\frac{n}{q_{1}}}dt\\
&  \lesssim \left \Vert \left(  b\left(  \cdot \right)  -b_{B}\right)
\right \Vert _{L_{p_{2}}\left(  {\mathbb{R}^{n}}\right)  }r^{\frac{n}{q_{1}}%
}\int \limits_{2r}^{\infty}\left \Vert f\right \Vert _{L_{p_{1}}\left(  B\left(
x_{0},t\right)  \right)  }t^{-1-\frac{n}{q_{1}}}dt\\
&  \lesssim \left \Vert b\right \Vert _{LC_{p_{2},\lambda}^{\left \{
x_{0}\right \}  }}r^{\frac{n}{q}+n\lambda}\int \limits_{2r}^{\infty}\left(
1+\ln \frac{t}{r}\right)  \left \Vert f\right \Vert _{L_{p_{1}}\left(  B\left(
x_{0},t\right)  \right)  }t^{-1-\frac{n}{q_{1}}}dt\\
&  \lesssim \left \Vert b\right \Vert _{LC_{p_{2},\lambda}^{\left \{
x_{0}\right \}  }}r^{\frac{n}{q}}\int \limits_{2r}^{\infty}\left(  1+\ln \frac
{t}{r}\right)  t^{n\lambda-\frac{n}{q_{1}}-1}\left \Vert f\right \Vert
_{L_{p_{1}}\left(  B\left(  x_{0},t\right)  \right)  }dt.
\end{align*}
When $q_{1}<s$, by the Fubini's theorem, the Minkowski inequality, the
H\"{o}lder's inequality and from (\ref{c}), (\ref{314}), we get%
\begin{align*}
\left \Vert J_{3}\right \Vert _{L_{q}\left(  B\right)  } &  \leq \left(
\int \limits_{B}\left \vert \int \limits_{2r}^{\infty}\int \limits_{B\left(
x_{0},t\right)  }\left \vert f\left(  y\right)  \right \vert \left \vert b\left(
x\right)  -b_{B}\right \vert \left \vert \Omega \left(  x-y\right)  \right \vert
dy\frac{dt}{t^{n-\alpha+1}}\right \vert ^{q}dx\right)  ^{\frac{1}{q}}\\
&  \leq \int \limits_{2r}^{\infty}\int \limits_{B\left(  x_{0},t\right)
}\left \vert f\left(  y\right)  \right \vert \left \Vert \left(  b\left(
\cdot \right)  -b_{B}\right)  \Omega \left(  \cdot-y\right)  \right \Vert
_{L_{q}\left(  B\right)  }dy\frac{dt}{t^{n-\alpha+1}}\\
&  \leq \int \limits_{2r}^{\infty}\int \limits_{B\left(  x_{0},t\right)
}\left \vert f\left(  y\right)  \right \vert \left \Vert b\left(  \cdot \right)
-b_{B}\right \Vert _{L_{p_{2}}\left(  B\right)  }\left \Vert \Omega \left(
\cdot-y\right)  \right \Vert _{L_{q_{1}}\left(  B\right)  }dy\frac
{dt}{t^{n-\alpha+1}}\\
&  \lesssim \left \Vert b\right \Vert _{LC_{p_{2},\lambda}^{\left \{
x_{0}\right \}  }}r^{\frac{n}{p_{2}}+n\lambda}\left \vert B\right \vert
^{\frac{1}{q_{1}}-\frac{1}{s}}\int \limits_{2r}^{\infty}\int \limits_{B\left(
x_{0},t\right)  }\left \vert f\left(  y\right)  \right \vert \left \Vert
\Omega \left(  \cdot-y\right)  \right \Vert _{L_{s}\left(  B\right)  }%
dy\frac{dt}{t^{n-\alpha+1}}\\
&  \lesssim \left \Vert b\right \Vert _{LC_{p_{2},\lambda}^{\left \{
x_{0}\right \}  }}r^{\frac{n}{q}-\frac{n}{s}+n\lambda}\int \limits_{2r}^{\infty
}\left \Vert f\right \Vert _{L_{1}\left(  B\left(  x_{0},t\right)  \right)
}\left \vert B\left(  x_{0},\frac{3}{2}t\right)  \right \vert ^{\frac{1}{s}%
}\frac{dt}{t^{n-\alpha+1}}\\
&  \lesssim \left \Vert b\right \Vert _{LC_{p_{2},\lambda}^{\left \{
x_{0}\right \}  }}r^{\frac{n}{q}-\frac{n}{s}+n\lambda}\int \limits_{2r}^{\infty
}\left(  1+\ln \frac{t}{r}\right)  \left \Vert f\right \Vert _{L_{p_{1}}\left(
B\left(  x_{0},t\right)  \right)  }\frac{dt}{t^{\frac{n}{q_{1}}-\frac{n}{s}%
+1}}\\
&  \lesssim \left \Vert b\right \Vert _{LC_{p_{2},\lambda}^{\left \{
x_{0}\right \}  }}r^{\frac{n}{q}-\frac{n}{s}}\int \limits_{2r}^{\infty}\left(
1+\ln \frac{t}{r}\right)  t^{n\lambda-\frac{n}{q_{1}}+\frac{n}{s}-1}\left \Vert
f\right \Vert _{L_{p_{1}}\left(  B\left(  x_{0},t\right)  \right)  }dt.
\end{align*}
On the other hand, for $J_{4}$, when $s^{\prime}\leq p$, for $x\in B$, by the
Fubini's theorem, applying the H\"{o}lder's inequality and from (\ref{b}),
(\ref{c}) (\ref{e311}) we have

$\left \vert T_{\Omega,\alpha}\left(  \left(  b\left(  \cdot \right)
-b_{B}\right)  f_{2}\right)  \left(  x\right)  \right \vert \lesssim%
{\displaystyle \int \limits_{\left(  2B\right)  ^{C}}}
\left \vert b\left(  y\right)  -b_{B}\right \vert \left \vert \Omega \left(
x-y\right)  \right \vert \frac{\left \vert f\left(  y\right)  \right \vert
}{\left \vert x-y\right \vert ^{n-\alpha}}dy$

$\lesssim%
{\displaystyle \int \limits_{\left(  2B\right)  ^{C}}}
\left \vert b\left(  y\right)  -b_{B}\right \vert \left \vert \Omega \left(
x-y\right)  \right \vert \frac{\left \vert f\left(  y\right)  \right \vert
}{\left \vert x_{0}-y\right \vert ^{n-\alpha}}dy$

$\approx%
{\displaystyle \int \limits_{2r}^{\infty}}
{\displaystyle \int \limits_{2r<\left \vert x_{0}-y\right \vert <t}}
\left \vert b\left(  y\right)  -b_{B}\right \vert \left \vert \Omega \left(
x-y\right)  \right \vert \left \vert f\left(  y\right)  \right \vert dy\frac
{dt}{t^{n-\alpha+1}}$

$\lesssim%
{\displaystyle \int \limits_{2r}^{\infty}}
{\displaystyle \int \limits_{B\left(  x_{0},t\right)  }}
\left \vert b\left(  y\right)  -b_{B\left(  x_{0},t\right)  }\right \vert
\left \vert \Omega \left(  x-y\right)  \right \vert \left \vert f\left(  y\right)
\right \vert dy\frac{dt}{t^{n-\alpha+1}}$

$+%
{\displaystyle \int \limits_{2r}^{\infty}}
\left \vert b_{B\left(  x_{0},r\right)  }-b_{B\left(  x_{0},t\right)
}\right \vert
{\displaystyle \int \limits_{B\left(  x_{0},t\right)  }}
\left \vert \Omega \left(  x-y\right)  \right \vert \left \vert f\left(  y\right)
\right \vert dy\frac{dt}{t^{n-\alpha+1}}$

$\lesssim%
{\displaystyle \int \limits_{2r}^{\infty}}
\left \Vert \left(  b\left(  \cdot \right)  -b_{B\left(  x_{0},t\right)
}\right)  f\right \Vert _{L_{p}\left(  B\left(  x_{0},t\right)  \right)
}\left \Vert \Omega \left(  \cdot-y\right)  \right \Vert _{L_{s}\left(  B\left(
x_{0},t\right)  \right)  }\left \vert B\left(  x_{0},t\right)  \right \vert
^{1-\frac{1}{p}-\frac{1}{s}}\frac{dt}{t^{n-\alpha+1}}$

$+%
{\displaystyle \int \limits_{2r}^{\infty}}
\left \vert b_{B\left(  x_{0},r\right)  }-b_{B\left(  x_{0},t\right)
}\right \vert \left \Vert f\right \Vert _{L_{p_{1}}\left(  B\left(
x_{0},t\right)  \right)  }\left \Vert \Omega \left(  \cdot-y\right)  \right \Vert
_{L_{s}\left(  B\left(  x_{0},t\right)  \right)  }\left \vert B\left(
x_{0},t\right)  \right \vert ^{1-\frac{1}{p_{1}}-\frac{1}{s}}t^{\alpha-n-1}dt$

$\lesssim%
{\displaystyle \int \limits_{2r}^{\infty}}
\left \Vert \left(  b\left(  \cdot \right)  -b_{B\left(  x_{0},t\right)
}\right)  \right \Vert _{L_{p_{2}}\left(  B\left(  x_{0},t\right)  \right)
}\left \Vert f\right \Vert _{L_{p_{1}}\left(  B\left(  x_{0},t\right)  \right)
}t^{-1-\frac{n}{q}}dt$

$+\left \Vert b\right \Vert _{LC_{p_{2},\lambda}^{\left \{  x_{0}\right \}  }}%
{\displaystyle \int \limits_{2r}^{\infty}}
\left(  1+\ln \frac{t}{r}\right)  \left \Vert f\right \Vert _{L_{p_{1}}\left(
B\left(  x_{0},t\right)  \right)  }t^{-1-\frac{n}{q_{1}}+n\lambda}dt$

$\lesssim \left \Vert b\right \Vert _{LC_{p_{2},\lambda}^{\left \{  x_{0}\right \}
}}%
{\displaystyle \int \limits_{2r}^{\infty}}
\left(  1+\ln \frac{t}{r}\right)  \left \Vert f\right \Vert _{L_{p_{1}}\left(
B\left(  x_{0},t\right)  \right)  }t^{-1-\frac{n}{q_{1}}+n\lambda}dt.$

Then, we have%
\begin{align*}
\left \Vert J_{4}\right \Vert _{L_{q}\left(  B\right)  }  & =\left \Vert
T_{\Omega,\alpha}\left(  b\left(  \cdot \right)  -b_{B}\right)  f_{2}%
\right \Vert _{L_{q}\left(  B\right)  }\\
& \lesssim \left \Vert b\right \Vert _{LC_{p_{2},\lambda}^{\left \{
x_{0}\right \}  }}r^{\frac{n}{q}}%
{\displaystyle \int \limits_{2r}^{\infty}}
\left(  1+\ln \frac{t}{r}\right)  t^{n\lambda-\frac{n}{q_{1}}-1}\left \Vert
f\right \Vert _{L_{p_{1}}\left(  B\left(  x_{0},t\right)  \right)  }dt.
\end{align*}

When $q<s$, by the Fubini's theorem, the Minkowski inequality, the
H\"{o}lder's inequality and from (\ref{b}), (\ref{c}), (\ref{314}) we have%
\begin{align*}
\left \Vert J_{4}\right \Vert _{L_{q}\left(  B\right)  } &  \leq \left(
\int \limits_{B}\left \vert \int \limits_{2r}^{\infty}\int \limits_{B\left(
x_{0},t\right)  }\left \vert b\left(  y\right)  -b_{B\left(  x_{0},t\right)
}\right \vert \left \vert f\left(  y\right)  \right \vert \left \vert
\Omega \left(  x-y\right)  \right \vert dy\frac{dt}{t^{n-\alpha+1}}\right \vert
^{q}dx\right)  ^{\frac{1}{q}}\\
&  +\left(  \int \limits_{B}\left \vert \int \limits_{2r}^{\infty}\left \vert
b_{B\left(  x_{0},r\right)  }-b_{B\left(  x_{0},t\right)  }\right \vert
\int \limits_{B\left(  x_{0},t\right)  }\left \vert f\left(  y\right)
\right \vert \left \vert \Omega \left(  x-y\right)  \right \vert dy\frac
{dt}{t^{n-\alpha+1}}\right \vert ^{q}dx\right)  ^{\frac{1}{q}}\\
&  \lesssim \int \limits_{2r}^{\infty}\int \limits_{B\left(  x_{0},t\right)
}\left \vert b\left(  y\right)  -b_{B\left(  x_{0},t\right)  }\right \vert
\left \vert f\left(  y\right)  \right \vert \left \Vert \Omega \left(
\cdot-y\right)  \right \Vert _{L_{q}\left(  B\left(  x_{0},t\right)  \right)
}dy\frac{dt}{t^{n-\alpha+1}}\\
&  +\int \limits_{2r}^{\infty}\left \vert b_{B\left(  x_{0},r\right)
}-b_{B\left(  x_{0},t\right)  }\right \vert \int \limits_{B\left(
x_{0},t\right)  }\left \vert f\left(  y\right)  \right \vert \left \Vert
\Omega \left(  \cdot-y\right)  \right \Vert _{L_{q}\left(  B\left(
x_{0},t\right)  \right)  }dy\frac{dt}{t^{n-\alpha+1}}\\
&  \lesssim \left \vert B\right \vert ^{\frac{1}{q}-\frac{1}{s}}\int
\limits_{2r}^{\infty}\int \limits_{B\left(  x_{0},t\right)  }\left \vert
b\left(  y\right)  -b_{B\left(  x_{0},t\right)  }\right \vert \left \vert
f\left(  y\right)  \right \vert \left \Vert \Omega \left(  \cdot-y\right)
\right \Vert _{L_{s}\left(  B\left(  x_{0},t\right)  \right)  }dy\frac
{dt}{t^{n-\alpha+1}}\\
&  +\left \vert B\right \vert ^{\frac{1}{q}-\frac{1}{s}}\int \limits_{2r}%
^{\infty}\left \vert b_{B\left(  x_{0},r\right)  }-b_{B\left(  x_{0},t\right)
}\right \vert \int \limits_{B\left(  x_{0},t\right)  }\left \vert f\left(
y\right)  \right \vert \left \Vert \Omega \left(  \cdot-y\right)  \right \Vert
_{L_{s}\left(  B\left(  x_{0},t\right)  \right)  }dy\frac{dt}{t^{n-\alpha+1}%
}\\
&  \lesssim r^{\frac{n}{q}-\frac{n}{s}}\int \limits_{2r}^{\infty}\left \Vert
\left(  b\left(  \cdot \right)  -b_{B\left(  x_{0},t\right)  }\right)
\right \Vert _{L_{p_{2}}\left(  B\left(  x_{0},t\right)  \right)  }\left \Vert
f\right \Vert _{L_{p_{1}}\left(  B\left(  x_{0},t\right)  \right)  }\left \vert
B\left(  x_{0},t\right)  \right \vert ^{1-\frac{1}{p}}\left \vert B\left(
x_{0},\frac{3}{2}t\right)  \right \vert ^{\frac{1}{s}}\frac{dt}{t^{n-\alpha+1}%
}\\
&  +r^{\frac{n}{q}-\frac{n}{s}}\int \limits_{2r}^{\infty}\left \vert b_{B\left(
x_{0},r\right)  }-b_{B\left(  x_{0},t\right)  }\right \vert \left \Vert
f\right \Vert _{L_{p_{1}}\left(  B\left(  x_{0},t\right)  \right)  }\left \vert
B\left(  x_{0},\frac{3}{2}t\right)  \right \vert ^{\frac{1}{s}}\frac
{dt}{t^{\frac{n}{p_{1}}-\alpha+1}}\\
&  \lesssim r^{\frac{n}{q}-\frac{n}{s}}\left \Vert b\right \Vert _{LC_{p_{2}%
,\lambda}^{\left \{  x_{0}\right \}  }}\int \limits_{2r}^{\infty}\left(
1+\ln \frac{t}{r}\right)  t^{n\lambda-\frac{n}{q_{1}}+\frac{n}{s}-1}\left \Vert
f\right \Vert _{L_{p_{1}}\left(  B\left(  x_{0},t\right)  \right)  }dt.
\end{align*}

By combining the above estimates, we complete the proof of Lemma \ref{Lemma 5}.
\end{proof}

Now we can give the following theorem (our main result).

\begin{theorem}
\label{teo15}Let $x_{0}\in{\mathbb{R}^{n}}$, $\Omega \in L_{s}(S^{n-1})$,
$1<s\leq \infty$, be homogeneous of degree zero. Let $T_{\Omega,\alpha}$ be a
linear operator satisfying condition (\ref{e1}) and bounded from
$L_{p}({\mathbb{R}^{n}})$ to $L_{q}({\mathbb{R}^{n}})$. Let $0<\alpha<n$,
$1<p<\frac{n}{\alpha}$, $b\in LC_{p_{2},\lambda}^{\left \{  x_{0}\right \}
}\left(
\mathbb{R}
^{n}\right)  $, $0\leq \lambda<\frac{1}{n}$, $\frac{1}{p}=\frac{1}{p_{1}}%
+\frac{1}{p_{2}}$, $\frac{1}{q}=\frac{1}{p}-\frac{\alpha}{n}$, $\frac{1}%
{q_{1}}=\frac{1}{p_{1}}-\frac{\alpha}{n}$.

Let also, for $s^{\prime}\leq p$ the pair $(\varphi_{1},\varphi_{2})$
satisfies the condition%
\begin{equation}
\int \limits_{r}^{\infty}\left(  1+\ln \frac{t}{r}\right)  \frac
{\operatorname*{essinf}\limits_{t<\tau<\infty}\varphi_{1}(x_{0},\tau
)\tau^{\frac{n}{p_{1}}}}{t^{\frac{n}{q_{1}}+1-n\lambda}}dt\leq C\, \varphi
_{2}(x_{0},r),\label{47}%
\end{equation}
and for $q_{1}<s$ the pair $(\varphi_{1},\varphi_{2})$ satisfies the condition%
\begin{equation}
\int \limits_{r}^{\infty}\left(  1+\ln \frac{t}{r}\right)  \frac
{\operatorname*{essinf}\limits_{t<\tau<\infty}\varphi_{1}(x_{0},\tau
)\tau^{\frac{n}{p_{1}}}}{t^{\frac{n}{q_{1}}-\frac{n}{s}+1-n\lambda}}dt\leq C\,
\varphi_{2}(x_{0},r)r^{\frac{n}{s}},\label{48}%
\end{equation}
where $C$ does not depend on $r$.

Then, the operator $T_{\Omega,b,\alpha}$ is bounded from $LM_{p_{1}%
,\varphi_{1}}^{\{x_{0}\}}$ to $LM_{q,\varphi_{2}}^{\{x_{0}\}}$. Moreover,
\[
\left \Vert T_{\Omega,b,\alpha}\right \Vert _{LM_{q,\varphi_{2}}^{\{x_{0}\}}%
}\lesssim \left \Vert b\right \Vert _{LC_{p_{2},\lambda}^{\left \{  x_{0}\right \}
}}\left \Vert f\right \Vert _{LM_{p_{1},\varphi_{1}}^{\{x_{0}\}}}.
\]

\end{theorem}

\begin{proof}
The statement of Theorem \ref{teo15} follows by Lemma \ref{Lemma 5} and
Theorem \ref{teo13} in the same manner as in the proof of Theorem \ref{teo9}.
\end{proof}

For the sublinear commutator of the fractional maximal operator with rough
kernel which is defined as follows%

\[
M_{\Omega,b,\alpha}\left(  f\right)  (x)=\sup_{t>0}|B(x,t)|^{-1+\frac{\alpha
}{n}}\int \limits_{B(x,t)}\left \vert b\left(  x\right)  -b\left(  y\right)
\right \vert \left \vert \Omega \left(  x-y\right)  \right \vert |f(y)|dy
\]
by Theorem \ref{teo15} we get the following new result.

\begin{corollary}
Suppose that $x_{0}\in{\mathbb{R}^{n}}$, $\Omega \in L_{s}(S^{n-1})$,
$1<s\leq \infty$, is homogeneous of degree zero. Let $0<\alpha<n$,
$1<p<\frac{n}{\alpha}$, $b\in LC_{p_{2},\lambda}^{\left \{  x_{0}\right \}
}\left(
\mathbb{R}
^{n}\right)  $, $0\leq \lambda<\frac{1}{n}$, $\frac{1}{p}=\frac{1}{p_{1}}%
+\frac{1}{p_{2}}$, $\frac{1}{q}=\frac{1}{p}-\frac{\alpha}{n}$, $\frac{1}%
{q_{1}}=\frac{1}{p_{1}}-\frac{\alpha}{n}$. Let also, for $s^{\prime}\leq p$
the pair $(\varphi_{1},\varphi_{2})$ satisfies the condition (\ref{47}) and
for $q_{1}<s$ the pair $(\varphi_{1},\varphi_{2})$ satisfies the condition
(\ref{48}). Then, the operators $M_{\Omega,b,\alpha}$ and $[b,\overline
{T}_{\Omega,\alpha}]$ are bounded from $LM_{p_{1},\varphi_{1}}^{\{x_{0}\}}$ to
$LM_{q,\varphi_{2}}^{\{x_{0}\}}$.
\end{corollary}

For the sublinear commutator of the fractional maximal operator is defined as follows%

\[
M_{b,\alpha}\left(  f\right)  (x)=\sup_{t>0}|B(x,t)|^{-1+\frac{\alpha}{n}}%
\int \limits_{B(x,t)}\left \vert b\left(  x\right)  -b\left(  y\right)
\right \vert |f(y)|dy
\]
by Theorem \ref{teo15} we get the following new result.

\begin{corollary}
Let $x_{0}\in{\mathbb{R}^{n}}$, $0<\alpha<n$, $1<p<\frac{n}{\alpha}$, $b\in
LC_{p_{2},\lambda}^{\left \{  x_{0}\right \}  }\left(
\mathbb{R}
^{n}\right)  $, $0\leq \lambda<\frac{1}{n}$, $\frac{1}{p}=\frac{1}{p_{1}}%
+\frac{1}{p_{2}}$, $\frac{1}{q}=\frac{1}{p}-\frac{\alpha}{n}$, $\frac{1}%
{q_{1}}=\frac{1}{p_{1}}-\frac{\alpha}{n}$ and the pair $(\varphi_{1}%
,\varphi_{2})$ satisfies condition (\ref{47}). Then, the operators
$M_{b,\alpha}$ and $[b,\overline{T}_{\alpha}]$ are bounded from $LM_{p_{1}%
,\varphi_{1}}^{\{x_{0}\}}$ to $LM_{q,\varphi_{2}}^{\{x_{0}\}}$.
\end{corollary}

Now using above results, we also obtain the boundedness of the operator
$T_{\Omega,b,\alpha}$ on the generalized vanishing local Morrey spaces
$VLM_{p,\varphi}^{\left \{  x_{0}\right \}  }$.

\begin{theorem}
\label{teo11}(Our main result) Let $x_{0}\in{\mathbb{R}^{n}}$, $\Omega \in
L_{s}(S^{n-1})$, $1<s\leq \infty$, be homogeneous of degree zero. Let
$0<\alpha<n$, $1<p<\frac{n}{\alpha}$, $b\in LC_{p_{2},\lambda}^{\left \{
x_{0}\right \}  }\left(
\mathbb{R}
^{n}\right)  $, $0\leq \lambda<\frac{1}{n}$, $\frac{1}{p}=\frac{1}{p_{1}}%
+\frac{1}{p_{2}}$, $\frac{1}{q}=\frac{1}{p}-\frac{\alpha}{n}$, $\frac{1}%
{q_{1}}=\frac{1}{p_{1}}-\frac{\alpha}{n}$, and $T_{\Omega,\alpha}$ is a linear
operator satisfying condition (\ref{e1}) and bounded from $L_{p}%
({\mathbb{R}^{n}})$ to $L_{q}({\mathbb{R}^{n}})$. Let for $s^{\prime}\leq p$
the pair $(\varphi_{1},\varphi_{2})$ satisfies conditions (\ref{2})-(\ref{3})
and%
\begin{equation}
\int \limits_{r}^{\infty}\left(  1+\ln \frac{t}{r}\right)  \varphi_{1}\left(
x_{0},t\right)  \frac{t^{\frac{n}{p_{1}}}}{t^{\frac{n}{q_{1}}+1-n\lambda}%
}dt\leq C_{0}\varphi_{2}\left(  x_{0},r\right)  ,\label{7*}%
\end{equation}
where $C_{0}$ does not depend on $r>0$,%
\begin{equation}
\lim_{r\rightarrow0}\frac{\ln \frac{1}{r}}{\varphi_{2}(x_{0},r)}=0\label{4.7}%
\end{equation}
and
\begin{equation}
c_{\delta}:=%
{\displaystyle \int \limits_{\delta}^{\infty}}
\left(  1+\ln \left \vert t\right \vert \right)  \varphi_{1}\left(
x_{0},t\right)  \frac{t^{\frac{n}{p_{1}}}}{t^{\frac{n}{q_{1}}+1-n\lambda}%
}dt<\infty \label{6*}%
\end{equation}
for every $\delta>0$, and for $q_{1}<s$ the pair $(\varphi_{1},\varphi_{2})$
satisfies conditions (\ref{2})-(\ref{3}) and also%
\begin{equation}
\int \limits_{r}^{\infty}\left(  1+\ln \frac{t}{r}\right)  \varphi_{1}\left(
x_{0},t\right)  \frac{t^{\frac{n}{p_{1}}}}{t^{\frac{n}{q_{1}}-\frac{n}%
{s}+1-n\lambda}}dt\leq C_{0}\varphi_{2}(x_{0},r)r^{\frac{n}{s}},\label{9*}%
\end{equation}
where $C_{0}$ does not depend on $r>0$,%
\[
\lim_{r\rightarrow0}\frac{\ln \frac{1}{r}}{\varphi_{2}(x_{0},r)}=0
\]
and%
\begin{equation}
c_{\delta^{\prime}}:=%
{\displaystyle \int \limits_{\delta^{\prime}}^{\infty}}
\left(  1+\ln \left \vert t\right \vert \right)  \varphi_{1}\left(
x_{0},t\right)  \frac{t^{\frac{n}{p_{1}}}}{t^{\frac{n}{q_{1}}-\frac{n}%
{s}+1-n\lambda}}dt<\infty \label{8*}%
\end{equation}
for every $\delta^{\prime}>0$.

Then the operator $T_{\Omega,b,\alpha}$ is bounded from $VLM_{p_{1}%
,\varphi_{1}}^{\left \{  x_{0}\right \}  }$ to$VLM_{q,\varphi_{2}}^{\left \{
x_{0}\right \}  }$. Moreover,%
\begin{equation}
\left \Vert T_{\Omega,b,\alpha}f\right \Vert _{VLM_{q,\varphi_{2}}^{\left \{
x_{0}\right \}  }}\lesssim \left \Vert b\right \Vert _{LC_{p_{2},\lambda
}^{\left \{  x_{0}\right \}  }}\left \Vert f\right \Vert _{VLM_{p_{1},\varphi_{1}%
}^{\left \{  x_{0}\right \}  }}.\label{10*}%
\end{equation}

\end{theorem}

\begin{proof}
The norm inequality having already been provided by Theorem \ref{teo15}, we
only have to prove the implication%
\begin{equation}
\lim \limits_{r\rightarrow0}\frac{r^{-\frac{n}{p_{1}}}\Vert f\Vert_{L_{p_{1}%
}(B(x_{0},r))}}{\varphi_{1}(x_{0},r)}=0\text{ implies }\lim
\limits_{r\rightarrow0}\frac{r^{-\frac{n}{q}}\left \Vert T_{\Omega,b,\alpha
}f\right \Vert _{L_{q}\left(  B\left(  x_{0},r\right)  \right)  }}{\varphi
_{2}(x_{0},r)}=0.\label{12*}%
\end{equation}

To show that%
\[
\frac{r^{-\frac{n}{q}}\left \Vert T_{\Omega,b,\alpha}f\right \Vert
_{L_{q}\left(  B\left(  x_{0},r\right)  \right)  }}{\varphi_{2}(x_{0}%
,r)}<\epsilon \text{ for small }r,
\]
we use the estimate (\ref{200}):%
\[
\frac{r^{-\frac{n}{q}}\left \Vert T_{\Omega,b,\alpha}f\right \Vert
_{L_{q}\left(  B\left(  x_{0},r\right)  \right)  }}{\varphi_{2}(x_{0}%
,r)}\lesssim \frac{\Vert b\Vert_{LC_{p_{2},\lambda}^{\{x_{0}\}}}}{\varphi
_{2}(x_{0},r)}\,
{\displaystyle \int \limits_{r}^{\infty}}
\left(  1+\ln \frac{t}{r}\right)  t^{n\lambda-\frac{n}{q_{1}}-1}\Vert
f\Vert_{L_{p_{1}}(B(x_{0},t))}dt.
\]

We take $r<\delta_{0}$, where $\delta_{0}$ will be chosen small enough and
split the integration:%
\begin{equation}
\frac{r^{-\frac{n}{q}}\left \Vert T_{\Omega,b,\alpha}f\right \Vert
_{L_{q}\left(  B\left(  x_{0},r\right)  \right)  }}{\varphi_{2}(x_{0},r)}\leq
C\left[  I_{\delta_{0}}\left(  x_{0},r\right)  +J_{\delta_{0}}\left(
x_{0},r\right)  \right]  ,\label{13*}%
\end{equation}
where $\delta_{0}>0$ (we may take $\delta_{0}<1$), and
\[
I_{\delta_{0}}\left(  x_{0},r\right)  :=\frac{1}{\varphi_{2}(x_{0},r)}%
{\displaystyle \int \limits_{r}^{\delta_{0}}}
\left(  1+\ln \frac{t}{r}\right)  t^{n\lambda-\frac{n}{q_{1}}-1}\left \Vert
f\right \Vert _{L_{p_{1}}\left(  B\left(  x_{0},t\right)  \right)  }dt,
\]
and%
\[
J_{\delta_{0}}\left(  x_{0},r\right)  :=\frac{1}{\varphi_{2}(x_{0},r)}%
{\displaystyle \int \limits_{\delta_{0}}^{\infty}}
\left(  1+\ln \frac{t}{r}\right)  t^{n\lambda-\frac{n}{q_{1}}-1}\left \Vert
f\right \Vert _{L_{p_{1}}\left(  B\left(  x_{0},t\right)  \right)  }dt
\]
and $r<\delta_{0}$. Now we choose any fixed $\delta_{0}>0$ such that%
\[
\frac{t^{-\frac{n}{p_{1}}}\left \Vert f\right \Vert _{L_{p_{1}}\left(  B\left(
x_{0},t\right)  \right)  }}{\varphi_{1}(x_{0},t)}<\frac{\epsilon}{2CC_{0}%
},\qquad t\leq \delta_{0},
\]
where $C$ and $C_{0}$ are constants from (\ref{7*}) and (\ref{13*}). This
allows to estimate the first term uniformly in $r\in \left(  0,\delta
_{0}\right)  $:%
\[
CI_{\delta_{0}}\left(  x_{0},r\right)  <\frac{\epsilon}{2},\qquad
0<r<\delta_{0}.
\]

For the second term, writing $1+\ln \frac{t}{r}\leq1+\left \vert \ln
t\right \vert +\ln \frac{1}{r}$, we obtain%
\[
J_{\delta_{0}}\left(  x_{0},r\right)  \leq \frac{c_{\delta_{0}}+\widetilde
{c_{\delta_{0}}}\ln \frac{1}{r}}{\varphi_{2}(x_{0},r)}\left \Vert f\right \Vert
_{LM_{p_{1},\varphi_{1}}^{\left \{  x_{0}\right \}  }},
\]
where $c_{\delta_{0}}$ is the constant from (\ref{6*}) with $\delta=\delta
_{0}$ and $\widetilde{c_{\delta_{0}}}$ is a similar constant with omitted
logarithmic factor in the integrand. Then, by (\ref{4.7}) we can choose small
enough $r$ such that%
\[
J_{\delta_{0}}\left(  x_{0},r\right)  <\frac{\epsilon}{2},
\]
which completes the proof of (\ref{12*}).

For the case of $q_{1}<s$, we can also use the same method, so we omit the details.
\end{proof}

\begin{remark}
Conditions (\ref{6*}) and (\ref{8*}) are not needed in the case when
$\varphi(x_{0},r)$ does not depend on $x_{0}$, since (\ref{6*}) follows from
(\ref{7*}) and similarly, (\ref{8*}) follows from (\ref{9*}) in this case.
\end{remark}

\begin{corollary}
Suppose that $x_{0}\in{\mathbb{R}^{n}}$, $\Omega \in L_{s}(S^{n-1})$,
$1<s\leq \infty$, is homogeneous of degree zero. Let $0<\alpha<n$,
$1<p<\frac{n}{\alpha}$, $b\in LC_{p_{2},\lambda}^{\left \{  x_{0}\right \}
}\left(
\mathbb{R}
^{n}\right)  $, $0\leq \lambda<\frac{1}{n}$, $\frac{1}{p}=\frac{1}{p_{1}}%
+\frac{1}{p_{2}}$, $\frac{1}{q}=\frac{1}{p}-\frac{\alpha}{n}$, $\frac{1}%
{q_{1}}=\frac{1}{p_{1}}-\frac{\alpha}{n}$. If for $s^{\prime}\leq p$ the pair
$\left(  \varphi_{1},\varphi_{2}\right)  $ satisfies conditions (\ref{2}%
)-(\ref{3})-(\ref{4.7}) and (\ref{6*})-(\ref{7*}) and for $q_{1}<s$ the pair
$\left(  \varphi_{1},\varphi_{2}\right)  $ satisfies conditions (\ref{2}%
)-(\ref{3})-(\ref{4.7}) and (\ref{8*})-(\ref{9*}). Then, the operators
$M_{\Omega,b,\alpha}$ and $[b,\overline{T}_{\Omega,\alpha}]$ are bounded from
$VLM_{p_{1},\varphi_{1}}^{\left \{  x_{0}\right \}  }$ to$VLM_{q,\varphi_{2}%
}^{\left \{  x_{0}\right \}  }$.
\end{corollary}

In the case of $q=\infty$ by Theorem \ref{teo11}, we get

\begin{corollary}
Let $x_{0}\in{\mathbb{R}^{n}}$, $0<\alpha<n$, $1<p<\frac{n}{\alpha}$, $b\in
LC_{p_{2},\lambda}^{\left \{  x_{0}\right \}  }\left(
\mathbb{R}
^{n}\right)  $, $0\leq \lambda<\frac{1}{n}$, $\frac{1}{p}=\frac{1}{p_{1}}%
+\frac{1}{p_{2}}$, $\frac{1}{q}=\frac{1}{p}-\frac{\alpha}{n}$, $\frac{1}%
{q_{1}}=\frac{1}{p_{1}}-\frac{\alpha}{n}$ and the pair $(\varphi_{1}%
,\varphi_{2})$ satisfies conditions (\ref{2})-(\ref{3})-(\ref{4.7}) and
(\ref{6*})-(\ref{7*}). Then the operators $M_{b,\alpha}$ and $[b,\overline
{T}_{\alpha}]$ are bounded from $VLM_{p_{1},\varphi_{1}}^{\left \{
x_{0}\right \}  }$ to$VLM_{q,\varphi_{2}}^{\left \{  x_{0}\right \}  }$.
\end{corollary}

\section{some applications}

In this section, we give the applications of Theorem \ref{teo9}, Theorem
\ref{teo10}, Theorem \ref{teo15}, Theorem \ref{teo11} for the Marcinkiewicz operator.

\subsection{Marcinkiewicz Operator}

Let $S^{n-1}=\{x\in{\mathbb{R}^{n}}:|x|=1\}$ be the unit sphere in
${\mathbb{R}^{n}}$ equipped with the Lebesgue measure $d\sigma$. Suppose that
$\Omega$ satisfies the following conditions.

(a) $\Omega$ is the homogeneous function of degree zero on ${\mathbb{R}^{n}%
}\setminus \{0\}$, that is,
\[
\Omega(\mu x)=\Omega(x),~~\text{for any}~~ \mu>0,x\in{\mathbb{R}^{n}}%
\setminus \{0\}.
\]

(b) $\Omega$ has mean zero on $S^{n-1}$, that is,
\[
\int \limits_{S^{n-1}}\Omega(x^{\prime})d\sigma(x^{\prime})=0,
\]
where $x^{\prime}=\frac{x}{\left \vert x\right \vert }$ for any $x\neq0$.

(c) $\Omega \in Lip_{\gamma}(S^{n-1})$, $0<\gamma \leq1$, that is there exists a
constant $M>0$ such that,
\[
|\Omega(x^{\prime})-\Omega(y^{\prime})|\leq M|x^{\prime}-y^{\prime}|^{\gamma
}~~\text{for any}~~x^{\prime},y^{\prime}\in S^{n-1}.
\]

In 1958, Stein \cite{Stein58} defined the Marcinkiewicz integral of higher
dimension $\mu_{\Omega}$ as
\[
\mu_{\Omega}(f)(x)=\left(  \int \limits_{0}^{\infty}|F_{\Omega,t}%
(f)(x)|^{2}\frac{dt}{t^{3}}\right)  ^{1/2},
\]
where
\[
F_{\Omega,t}(f)(x)=\int \limits_{|x-y|\leq t}\frac{\Omega(x-y)}{|x-y|^{n-1}%
}f(y)dy.
\]

Since Stein's work in 1958, the continuity of Marcinkiewicz integral has been
extensively studied as a research topic and also provides useful tools in
harmonic analysis \cite{LuDingY, St, Stein93, Torch}.

The Marcinkiewicz operator is defined by (see \cite{TorWang})
\[
\mu_{\Omega,\alpha}(f)(x)=\left(  \int \limits_{0}^{\infty}|F_{\Omega,\alpha
,t}(f)(x)|^{2}\frac{dt}{t^{3}}\right)  ^{1/2},
\]
where
\[
F_{\Omega,\alpha,t}(f)(x)=\int \limits_{|x-y|\leq t}\frac{\Omega(x-y)}%
{|x-y|^{n-1-\alpha}}f(y)dy.
\]

Note that $\mu_{\Omega}f=\mu_{\Omega,0}f$.

The sublinear commutator of the operator $\mu_{\Omega,\alpha}$ is defined by
\[
\lbrack b,\mu_{\Omega,\alpha}](f)(x)=\left(
{\displaystyle \int \limits_{0}^{\infty}}
|F_{\Omega,\alpha,t,b}(f)(x)|^{2}\frac{dt}{t^{3}}\right)  ^{1/2},
\]
where
\[
F_{\Omega,\alpha,t,b}(f)(x)=%
{\displaystyle \int \limits_{|x-y|\leq t}}
\frac{\Omega(x-y)}{|x-y|^{n-1-\alpha}}[b(x)-b(y)]f(y)dy.
\]

We consider the space $H=\{h:\Vert h\Vert=(\int \limits_{0}^{\infty}%
|h(t)|^{2}\frac{dt}{t^{3}})^{1/2}<\infty \}$. Then, it is clear that
$\mu_{\Omega,\alpha}(f)(x)=\Vert F_{\Omega,\alpha,t}(x)\Vert$.

By the Minkowski inequality, we get
\[
\mu_{\Omega,\alpha}(f)(x)\leq \int \limits_{{\mathbb{R}^{n}}}\frac
{|\Omega(x-y)|}{|x-y|^{n-1-\alpha}}|f(y)|\left(  \int \limits_{|x-y|}^{\infty
}\frac{dt}{t^{3}}\right)  ^{1/2}dy\leq C\int \limits_{{\mathbb{R}^{n}}}%
\frac{\left \vert \Omega(x-y)\right \vert }{|x-y|^{n-\alpha}}|f(y)|dy.
\]
Thus, $\mu_{\Omega,\alpha}$ satisfies the condition (\ref{e1}). It is known
that $\mu_{\Omega,\alpha}$ is bounded from $L_{p}({\mathbb{R}^{n}})$ to
$L_{q}({\mathbb{R}^{n}})$ for $p>1$, and bounded from $L_{1}({\mathbb{R}^{n}%
})$ to $WL_{q}({\mathbb{R}^{n}})$ for $p=1$ (see \cite{TorWang}), then by
Theorems \ref{teo9}, \ref{teo10}, \ref{teo15} and \ref{teo11} we get

\begin{corollary}
Suppose that $x_{0}\in{\mathbb{R}^{n}}$, $\Omega \in L_{s}(S^{n-1})$,
$1<s\leq \infty$, is homogeneous of degree zero. Let $0<\alpha<n$, $1\leq
p<\frac{n}{\alpha}$ and $\frac{1}{q}=\frac{1}{p}-\frac{\alpha}{n}$. Let also,
for $s^{\prime}\leq p$, $p\neq1$, the pair $\left(  \varphi_{1},\varphi
_{2}\right)  $ satisfies condition (\ref{316}) and for $q<s$ the pair $\left(
\varphi_{1},\varphi_{2}\right)  $ satisfies condition (\ref{317}) and $\Omega$
satisfies conditions (a)--(c). Then the operator $\mu_{\Omega,\alpha}$ is
bounded from $LM_{p,\varphi_{1}}^{\{x_{0}\}}$ to $LM_{q,\varphi_{2}}%
^{\{x_{0}\}}$ for $p>1$ and from $LM_{1,\varphi_{1}}^{\{x_{0}\}}$ to
$WLM_{q,\varphi_{2}}^{\{x_{0}\}}$ for $p=1$.
\end{corollary}

\begin{corollary}
Suppose that $x_{0}\in{\mathbb{R}^{n}}$, $\Omega \in L_{s}(S^{n-1})$,
$1<s\leq \infty$, is homogeneous of degree zero. Let $0<\alpha<n$, $1\leq
p<\frac{n}{\alpha}$ and $\frac{1}{q}=\frac{1}{p}-\frac{\alpha}{n}$. Let also,
for $s^{\prime}\leq p$, $p\neq1$, the pair $\left(  \varphi_{1},\varphi
_{2}\right)  $ satisfies conditions (\ref{2})-(\ref{3}) and (\ref{6}%
)-(\ref{7}) and for $q<s$ the pair $\left(  \varphi_{1},\varphi_{2}\right)  $
satisfies conditions (\ref{2})-(\ref{3}) and (\ref{8})-(\ref{9}) and $\Omega$
satisfies conditions (a)--(c). Then the operator $\mu_{\Omega,\alpha}$ is
bounded from $VLM_{p,\varphi_{1}}^{\{x_{0}\}}$ to $VLM_{q,\varphi_{2}%
}^{\{x_{0}\}}$ for $p>1$ and from $VLM_{1,\varphi_{1}}^{\{x_{0}\}}$ to
$WVLM_{q,\varphi_{2}}^{\{x_{0}\}}$ for $p=1$.
\end{corollary}

\begin{corollary}
Suppose that $x_{0}\in{\mathbb{R}^{n}}$, $\Omega \in L_{s}(S^{n-1})$,
$1<s\leq \infty$, is homogeneous of degree zero. Let $0<\alpha<n$,
$1<p<\frac{n}{\alpha}$, $b\in LC_{p_{2},\lambda}^{\left \{  x_{0}\right \}
}\left(
\mathbb{R}
^{n}\right)  $, $0\leq \lambda<\frac{1}{n}$, $\frac{1}{p}=\frac{1}{p_{1}}%
+\frac{1}{p_{2}}$, $\frac{1}{q}=\frac{1}{p}-\frac{\alpha}{n}$, $\frac{1}%
{q_{1}}=\frac{1}{p_{1}}-\frac{\alpha}{n}$. Let also, for $s^{\prime}\leq p$
the pair $(\varphi_{1},\varphi_{2})$ satisfies condition (\ref{47}) and for
$q_{1}<s$ the pair $(\varphi_{1},\varphi_{2})$ satisfies condition (\ref{48})
and $\Omega$ satisfies conditions (a)--(c). Then, the operator $[b,\mu
_{\Omega,\alpha}]$ is bounded from $LM_{p_{1},\varphi_{1}}^{\{x_{0}\}}$ to
$LM_{q,\varphi_{2}}^{\{x_{0}\}}$.
\end{corollary}

\begin{corollary}
Suppose that $x_{0}\in{\mathbb{R}^{n}}$, $\Omega \in L_{s}(S^{n-1})$,
$1<s\leq \infty$, is homogeneous of degree zero. Let $0<\alpha<n$,
$1<p<\frac{n}{\alpha}$, $b\in LC_{p_{2},\lambda}^{\left \{  x_{0}\right \}
}\left(
\mathbb{R}
^{n}\right)  $, $0\leq \lambda<\frac{1}{n}$, $\frac{1}{p}=\frac{1}{p_{1}}%
+\frac{1}{p_{2}}$, $\frac{1}{q}=\frac{1}{p}-\frac{\alpha}{n}$, $\frac{1}%
{q_{1}}=\frac{1}{p_{1}}-\frac{\alpha}{n}$. Let also, for $s^{\prime}\leq p$
the pair $\left(  \varphi_{1},\varphi_{2}\right)  $ satisfies conditions
(\ref{2})-(\ref{3})-(\ref{4.7}) and (\ref{6*})-(\ref{7*}) and for $q_{1}<s$
the pair $\left(  \varphi_{1},\varphi_{2}\right)  $ satisfies conditions
(\ref{2})-(\ref{3})-(\ref{4.7}) and (\ref{8*})-(\ref{9*}) and $\Omega$
satisfies conditions (a)--(c). Then, the operator $[b,\mu_{\Omega,\alpha}] $
is bounded from $VLM_{p_{1},\varphi_{1}}^{\left \{  x_{0}\right \}  }$
to$VLM_{q,\varphi_{2}}^{\left \{  x_{0}\right \}  }$.
\end{corollary}


\begin{thebibliography}{99}                                                                                               %
\bibitem {Adams}D.R. Adams, A note on Riesz potentials, Duke Math. J., 42
(1975), 765-778.

\bibitem {Akbulut-Kuzu}A. Akbulut, O. Kuzu, Marcinkiewicz integrals with rough
kernel associated with Schr\"{o}dinger operator on vanishing generalized
Morrey Spaces. Azerb. J. Math., 4 (1) (2014), 40-54.

\bibitem {AlvLanLakey}J. Alvarez, M. Guzman-Partida, J. Lakey, Spaces of
bounded $\lambda$-central mean oscillation, Morrey spaces, and $\lambda
$-central Carleson measures, Collect. Math., 51 (2000), 1-47.

\bibitem {Beurl}A. Beurling, Construction and analysis of some convolution
algebras, Ann. Inst. Fourier (Grenoble), 14 (1964), 1--32.

\bibitem {BGGS}A.S. Balakishiyev, V.S. Guliyev, F. Gurbuz and A. Serbetci,
Sublinear operators with rough kernel generated by Calderon-Zygmund operators
and their commutators on generalized local Morrey spaces, Journal of
Inequalities and Applications 2015, 2015:61. doi:10.1186/s13660-015-0582-y.

\bibitem {Caf}L. Caffarelli, Elliptic second order equations, Rend. Semin.
Math. Fis. Milano, 58 (1990), 253-284.

\bibitem {Cao-Chen}X.N. Cao, D.X. Chen, The boundedness of Toeplitz-type
operators on vanishing Morrey spaces. Anal. Theory Appl. 27 (2011), 309-319.

\bibitem {Chanillo}S. Chanillo, A note on commutators,\ Indiana Univ. Math.
J., \textbf{31} (1), (1982 ), 7-16.

\bibitem {ChanilloWW}S. Chanillo, D.K. Watson, R.L. Wheeden,\textit{\ }Some
integral and maximal operators related to starlike sets. Studia Math., 107
(1993), 223-255.

\bibitem {ChFra}F. Chiarenza, M. Frasca,\textit{\ }Morrey spaces and
Hardy-Littlewood maximal function, Rend. Mat., 7 (1987), 273-279.

\bibitem {ChFraL1}F. Chiarenza, M. Frasca, P. Longo, Interior $W^{2,p}%
$-estimates for nondivergence elliptic equations with discontinuous
coefficients, Ricerche Mat., 40 (1991), 149-168.

\bibitem {ChFraL2}F. Chiarenza, M. Frasca, P. Longo, $W^{2,p}$-solvability of
Dirichlet problem for nondivergence elliptic equations with VMO coefficients,
Trans. Amer. Math. Soc., 336 (1993), 841-853.

\bibitem {CRW}R.R. Coifman, R. Rochberg, G. Weiss, Factorization theorems for
Hardy spaces in several variables, Ann. of Math., 103 (3) (1976), 611-635.

\bibitem {CLMS}R.R. Coifman, P. Lions, Y. Meyer, S. Semmes, Compensated
compactness and Hardy spaces. J. Math. Pures Appl. 72 (1993), 247-286.

\bibitem {CM}R.R. Coifman, Y. Meyer, Au del\`{a} des Op\'{e}rateurs
Pseudo-Diff\'{e}rentiels, Ast\'{e}risque 57. Soci\'{e}t\'{e} Math\'{e}matique
de France, Paris, 1978, 185 pp.

\bibitem {Ding}Y. Ding, D.C. Yang, Z. Zhou, Boundedness of sublinear operators
and commutators on $L^{p,\omega}(%
\mathbb{R}
^{n})$, Yokohama Math. J., 46 (1998), 15-27.

\bibitem {DingLu2}Y. Ding and S.Z. Lu, Higher order commutators for a class of
rough operators, Ark. Mat., 37 (1999), 33-44.

\bibitem {DingLu3}Y. Ding and S.Z. Lu, Homogeneous fractional integrals on
Hardy spaces. Tohoku Math. J. 52 (2000), 153-162.

\bibitem {FazRag2}G. Di Fazio, M.A. Ragusa, Interior estimates in Morrey
spaces for strong solutions to nondivergence form equations with discontinuous
coefficients, J. Funct. Anal., 112 (1993), 241-256.

\bibitem {FazPalRag}G. Di Fazio, D.K. Palagachev and M.A. Ragusa, Global
Morrey regularity of strong solutions to the Dirichlet problem for elliptic
equations with discontinuous coefficients, J. Funct. Anal., 166 (1999), 179-196.

\bibitem {Feicht}H. Feichtinger, An elementary approach to Wiener's third
Tauberian theorem on Euclidean $n$-space, Proceedings, Conference at Cortona
1984, Sympos. Math., 29, Academic Press 1987.

\bibitem {FuLinLu}Z.W. Fu, Y. Lin, S.Z. Lu, $\lambda$-Central $BMO$\ estimates
for commutators of singular integral operators with rough kernel. Acta Math.
Sin. 24 (2008), 373-386.

\bibitem {GarRub}J. Garcia-Cuerva and J.L. Rubio de Francia, Weighted Norm
Inequalities and Related Topics, North-Holland Math. 16, Amsterdam, 1985.

\bibitem {GraLiYang}L. Grafakos, X.W. Li, D.C. Yang, Bilinear operators on
Herz-type Hardy spaces. Trans. Amer. Math. Soc., 350 (1998), 1249-1275.

\bibitem {Grafakos}L. Grafakos,\textit{\ }Classical and modern Fourier
analysis. Pearson Education. Inc. Upper Saddle River, New Jersey, 2004.

\bibitem {GulJIA}V.S. Guliyev,\textit{\ }Boundedness of the maximal, potential
and singular operators in the generalized Morrey spaces, J. Inequal. Appl.
2009, Art. ID 503948, 20 pp.

\bibitem {GulJMS2013}V.S. Guliyev, Generalized local Morrey spaces and
fractional integral operators with rough kernel, J. Math. Sci. (N.Y.), 193 (2)
(2013), 211-227.

\bibitem {Gurbuz}F. Gurbuz\textit{, }Boundedness of some potential type
sublinear operators and their commutators with rough kernels on generalized
local Morrey spaces $\left[  \text{\textit{Ph.D. thesis}}\right]  $, Ankara
University, Ankara, Turkey, 2015 (in Turkish).

\bibitem {Gurbuz2}F. Gurbuz\textit{, }Parabolic sublinear operators with rough
kernel generated by parabolic Calder\'{o}n-Zygmund operators and parabolic
local Campanato space estimates for their commutators on the parabolic
generalized local Morrey spaces, Open Math., 14 (2016), 300-323.

\bibitem {Gurbuz3}F. Gurbuz\textit{, }Parabolic sublinear operators with rough
kernel generated by parabolic fractional integral operators and parabolic
local Campanato space estimates for their commutators on the parabolic
generalized local Morrey spaces, Advan. Math. (China), in press.

\bibitem {Janson}S. Janson, Mean oscillation and commutators of singular
integral operators, Ark. Mat., 16 (1978), 263-270.

\bibitem {Karaman}T. Karaman, Boundedness of some classes of sublinear
operators on generalized weighted Morrey spaces and some applications $\left[
\text{\textit{Ph.D. thesis}}\right]  $, Ankara University, Ankara, Turkey,
2012 (in Turkish).

\bibitem {KATO}T. Kato\textit{, }Strong $L^{p}$\ solutions of the
Navier-Stokes equations in $%
\mathbb{R}
^{m}$\ with applications to weak solutions, Math. Z., 187 (1984), 471-480.

\bibitem {Li-Yang}X. Li and D.C. Yang, Boundedness of some sublinear operators
on Herz spaces, Illinois J. of Math., 40 (1996), 484-501.

\bibitem {L-SU}Y. Liang and W. Su, The relationship between fractal dimensions
of a type of fractal functions and the order of their fractional calculus,
Chaos, Solitions and Fractals, 34 (2007), 682-692.

\bibitem {LuYang1}S.Z. Lu and D.C. Yang, The central BMO spaces and
Littlewood-Paley operators, Approx. Theory Appl. (N.S.), 11 (1995), 72-94.

\bibitem {LLY}G. Lu, S.Z. Lu, D.C. Yang, Singular integrals and commutators on
homogeneous groups, Anal. Math., 28 (2002), 103-134.

\bibitem {LuDingY}S.Z. Lu, Y. Ding, D.Y. Yan, Singular integrals and related
topics, World Scientific Publishing, Singapore, 2006.

\bibitem {LuWu}S.Z. Lu and Q. Wu, \textit{\ }CBMO estimates for commutators
and multilinear singular integrals, Math. Nachr., 276 (2004), 75-88.

\bibitem {Mazzucato}A. Mazzucato, Besov-Morrey spaces:functions space theory
and applications to non-linear PDE., Trans. Amer. Math. Soc., 355 (2002), 1297-1364.

\bibitem {Miranda}C. Miranda,\textit{\ }Sulle equazioni ellittiche del secondo
ordine di tipo non variazionale, a coefficienti discontinui. Ann. Math. Pura E
Appl. 63 (4) (1963), 353-386.

\bibitem {Miz}T. Mizuhara, Boundedness of some classical operators on
generalized Morrey spaces, Harmonic Analysis (S. Igari, Editor), ICM 90
Satellite Proceedings, Springer - Verlag, Tokyo (1991), 183-189.

\bibitem {Morrey}C.B. Morrey, On the solutions of quasi-linear elliptic
partial differential equations,\ Trans. Amer. Math. Soc., 43 (1938), 126-166.

\bibitem {Muckenhoupt}B. Muckenhoupt, On certain singular integrals. Pacif. J.
Math., 10 (1960), 239-261.

\bibitem {Muckenhoupt and Wheeden}B. Muckenhoupt, R.L. Wheeden, Weighted norm
inequalities for singular and fractional integrals. Trans. Amer. Math. Soc.,
161 (1971), 249-258.

\bibitem {K-J}K. Oldham and J. Spanier\textit{, }The fractional calculus, New
York, Acedemic press, 1974.

\bibitem {Pal}D.K. Palagachev, L.G. Softova, Singular integral operators,
Morrey spaces and fine regularity of solutions to PDE's, Potential Anal., 20
(2004), 237-263.

\bibitem {Palus}M. Paluszynski, Characterization of the Besov spaces via the
commutator operator of Coifman, Rochberg and Weiss, Indiana Univ. Math. J.,
\textbf{44} (1995), 1-17.

\bibitem {Peetre}J. Peetre, On the theory of $M_{p,\lambda}$, J. Funct. Anal.,
4 (1969), 71-87.

\bibitem {PerRagSamWall}L.E. Persson, M.A. Ragusa, N. Samko, P. Wall,
Commutators of Hardy operators in vanishing Morrey spaces. AIP Conf. Proc.
1493, 859 (2012); http://dx.doi.org/10.1063/1.4765588.

\bibitem {RagusaJGlOpt}M.A. Ragusa, Commutators of fractional integral
operators on vanishing-Morrey spaces. J. Global Optim. 368 (40) (2008), 1-3.

\bibitem {Ruiz}A. Ruiz, L. Vega, On local regularity of Schr\"{o}dinger
equations, Int. Math. Res. Not., 1 (1993), 13-27.

\bibitem {Rzaev}R.M. Rzaev,\ On approximation of local summary functions by
singular integrals in terms of mean oscillation and some applications,
Preprint No. 1 of Inst. Physics of NAS of Azerb., 1992, pp. 1-43 (Russian).

\bibitem {N. Samko}N. Samko, Maximal, Potential and Singular Operators in
vanishing generalized Morrey Spaces. J. Global Optim. 2013, 1-15. DOI 10.1007/s10898-012-9997-x.

\bibitem {Shi}S.G. Shi, S.Z. Lu, A characterization of Campanato space via
commutator of fractional integral, J. Math. Anal. Appl., 419 (2014), 123-137.

\bibitem {Softova}L.G. Softova, Singular integrals and commutators in
generalized Morrey spaces, Acta Math. Sin. (Engl. Ser.), \textbf{22} (2006), 757-766.

\bibitem {SW}F. Soria, G. Weiss, A remark on singular integrals and power
weights, Indiana Univ. Math. J., 43 (1994) 187-204.

\bibitem {Stein58}E.M. Stein,\textit{\ }On the functions of Littlewood-Paley,
Lusin and Marcinkiewicz, Trans. Amer. Math. Soc. 88 (1958) 430-466.

\bibitem {St}E.M. Stein, Singular integrals and differentiability of
functions, Princeton University Press, Princeton, NJ, 1970.

\bibitem {Stein93}E.M. Stein, Harmonic Analysis: Real Variable Methods,
Orthogonality and Oscillatory Integrals, Princeton Univ. Press, Princeton NJ, 1993.

\bibitem {TaoShi}X.X. Tao, Y.L. Shi, Multilinear commutators of
Calder\'{o}n-Zygmund operator on $\lambda$-central Morrey spaces. Adv. Math.
(China), 40 (2011), 47-59.

\bibitem {Torch}A. Torchinsky, Real Variable Methods in Harmonic Analysis,
Pure and Applied Math. 123, Academic Press, New York, 1986.

\bibitem {TorWang}A. Torchinsky and S. Wang, A note on the Marcinkiewicz
integral, Colloq. Math. 60/61 (1990), 235-243.

\bibitem {Vitanza1}C. Vitanza, Functions with vanishing Morrey norm and
elliptic partial differential equations. In: Proceedings of Methods of Real
Analysis and Partial Differential Equations,Capri, pp. 147-150. Springer (1990).

\bibitem {Vitanza2}C. Vitanza, Regularity results for a class of elliptic
equations with coefficients in Morrey spaces. Ricerche di Matematica 42 (2)
(1993), 265-281.

\bibitem {Wiener1}N. Wiener, Generalized Harmonic Analysis, Acta Math., 55
(1930), 117-258.

\bibitem {Wiener2}N. Wiener, Tauberian theorems, Ann. Math., 33 (1932), 1-100.

\bibitem {Yu}X. Yu, S.Z. Lu, Boundedness for a class of fractional integrals
with rough kernel related to block spaces, Front. Math. China, 2015. DOI 10.1007/s11464-015-0499-2.
\end{thebibliography}
\end{document}